\documentclass[letter,11pt]{amsart}
\usepackage{amssymb} 
\usepackage{graphicx}
\usepackage{color}
\usepackage[font=footnotesize]{caption}
\usepackage{mathtools}
\usepackage{pinlabel}
\usepackage{graphicx}
\graphicspath{ {images/} }
\usepackage{subfig}

\usepackage{hyperref}
\usepackage{float}
\usepackage{calc}
\def\thetitle{{ On the finiteness property of hyperbolic simplicial actions: the right-angled Artin groups  and their extension graphs  }}
\hypersetup{
pdftitle=  \thetitle,
pdfauthor=  {Hyungryul Baik, Donggyun Seo, Hyunshik Shin}
}

\newtheorem{thmalpha}{Theorem}

\newtheorem{thm}{Theorem}[section]
\newtheorem{lem}[thm]{Lemma}
\newtheorem{cor}[thm]{Corollary}
\newtheorem{prop}[thm]{Proposition}

\theoremstyle{remark}

\newtheorem*{rem}{Remark}
\usepackage{tikz}
\usetikzlibrary{calc,decorations.markings}
\usetikzlibrary{shapes,snakes}

\theoremstyle{definition}
\newtheorem*{defn*}{Definition}

\title[extension graph]\thetitle

\author{Hyungryul Baik}
\address{Department of Mathematical Sciences, KAIST,  
291 Daehak-ro, Yuseong-gu, Daejeon 34141, South Korea }
\email{hrbaik@kaist.ac.kr}

\author{Donggyun Seo}
\address{Department of Mathematical Sciences, KAIST,  
291 Daehak-ro, Yuseong-gu, Daejeon 34141, South Korea }
\email{seodonggyun@kaist.ac.kr}

\author{Hyunshik Shin}
\address{%
Department of Mathematics and Statistics, 
Georgia State University,
25 Park Place, 14th Floor, 
Atlanta, GA 30303
}
\email{%
        hshin30@gsu.edu
}

\begin{document}
\maketitle

\begin{abstract}
We study the right-angled Artin group action on the extension graph. We show that this action satisfies a certain finiteness property, which is a variation of a condition introduced by Delzant and Bowditch. As an application we show that the asymptotic translation lengths of elements of a given right-angled Artin group are always rational and once the defining graph has girth at least 6, they have a common denominator. We construct explicit examples which show the denominator of the asymptotic translation length of such an action can be arbitrary. We also observe that if either an element has a small syllable length or the defining graph for the right-angled Artin group is a tree then the asymptotic translation lengths are integers. 
\end{abstract}



%
%

\section{Introduction} \label{sec:intro} 
When a group $G$ acts on a metric space $(X,d)$ by isometries, one can
define the asymptotic translation length of each element of $G$ as follows: 
$$ \tau(g) = \lim_{n\to\infty} \dfrac{d(x, g^nx)}{n} $$
where $g\in G$ and $x\in X$. One can easily see that the limit exists and does not depend on the choice of $x$ (see for instance Exercise 6.6 in Chapter II.6 of \cite{MR1744486}). Note that $\tau(\cdot)$ is homogeneous in the sense that $\tau(g^n) = n\tau(g)$ for all $g\in G, n \in \mathbb{Z}$. 

The asymptotic translation lengths have been studied by many authors for group actions which arise naturally in geometric topology.  In the case that $G$ is the mapping class group of a surface and $X$ is the curve complex, then the geometric/dynamical aspect of the asymptotic translation length has been studied in the literature. For instance, Masur-Minsky \cite{MasurMinsky99} showed that for a given mapping class $f$, $\tau(f)$ is positive if and only if $f$ is pseudo-Anosov (to show a pseudo-Anosov mapping class makes a definite asymptotic progress, they proved so-called the nesting lemma). The minimal asymptotic translation lengths for various subsets of the mapping class groups are studied in, for instance, \cite{GadreTsai11}, \cite{GadreHironakaKentLeininger13}, \cite{Valdivia17}, \cite{KinShin18}, \cite{BaikShin18}, \cite{BaikShinWu18}. 

In general, a simplicial group action on a simplicial graph with the edge metric may contain an irrational length element. For example, Conner \cite{MR1466819} found a polycyclic group whose action on its Cayley graph contains an irrational length element with respect to the word metric. On the other hand, Gromov \cite[Section 8.5.S]{MR919829} discovered that every hyperbolic group has a discrete rational length spectrum. More precisely, Gromov proved that for a group acting simplicially, properly, and cocompactly on a $\delta$-hyperbolic graph equipped with the edge metric, every element has a rational asymptotic translation length with the common denominator depending only on the action. Delzant \cite[Proposition 3.1(iii)]{MR1390660} gave another simple proof of Gromov's result. 

Later Delzant's method was adapted by Bowditch \cite[Theorem 1.4]{Bowditch08} which shows that the asymptotic translation lengths of elements of a given mapping class group are rational numbers with uniformly bounded denominator on the curve complex. Note that the set up of \cite{Bowditch08} is quite different from the one for Gromov or Delzant. While it is still true that curve complexes are $\delta$-hyperbolic \cite{MasurMinsky99}, mapping class groups are not hyperbolic, curve complexes are locally infinite, and the action is non-proper. 

We refine Bowditch's method and apply to another important player in the geometric group theory, the right-angled Artin groups. For the right-angled Artin groups, Kim--Koberda \cite{kim2013embedability} introduced the notion of the extension graph. For a finite simplicial graph $\Gamma$, the associated right-angled Artin group (RAAG) $A(\Gamma)$ acts on the extension graph $\Gamma^e$ by isometries which is a right action by conjugation. The extension graphs and right-angled Artin group actions on them share many similar properties with the curve graph and mapping class group actions. For more detail, see \cite{kim2013embedability}, \cite{kim2014geometry}, \cite{kim2014obstruction}, \cite{koberda2017geometry} for instance. We also give a brief review on this material in Section \ref{sec:prelim}.

We consider the asymptotic translation length of loxodromic elements of
$A(\Gamma)$ with respect to this action on $\Gamma^e$. Our main result is to show that the asymptotic translation lengths of loxodromic elements of $A(\Gamma)$ on $\Gamma^e$ are rational numbers (with uniformly bounded denominators in many cases) which is an analogue of the theorem of Bowditch \cite[Theorem 1.4]{Bowditch08}. Throughout the paper, we assume that all our graphs are connected unless specified otherwise.

\begin{thmalpha}[Main Theorem]\footnote{While the paper was being reviewed, A. Genevois independently proved a more general version of the first half of the main theorem in \cite{2022arXiv220906441G}. } \label{thm:main1}
Let $\Gamma$ be any finite connected simplicial graph. Then for the action of $A(\Gamma)$ on the extension graph $\Gamma^e$, all loxodromic elements have rational asymptotic translation lengths. If the graph $\Gamma$ has girth at least 6 in addition, the asymptotic translation length have a common denominator. \end{thmalpha} 

In fact, this is a special case of actions satisfying so-called \emph{$\kappa$-finiteness property}.
An axial subgraph of a loxodromic is a thickened geodesic axis that can be separated by finitely many vertices.
We define the width of an axial subgraph to be the minimum number of vertices needed to be removed to increase the number of connected components of the axial subgraph.

\begin{defn*}[Finiteness property]
Suppose a group $G$ acts simplicially on a $\delta$-hyperbolic graph $\mathcal{G}$.
\begin{enumerate}
\item The action of $G$ on $\mathcal{G}$ is said to have the \emph{finiteness property} if every loxodromic has an axial subgraph.
\item For an integer $\kappa \geq 1$, the action of $G$ on $\mathcal{G}$ is said to have the \emph{$\kappa$-finiteness property} if every loxodromic of $G$ has an axial subgraph of width at most $\kappa$.
\end{enumerate}
\end{defn*}

Let $\operatorname{Spec}(G, \mathcal{G})$ denote the spectrum of asymptotic translation lengths of all elements of $G$ and we call it the \emph{length spectrum} of $G$ on $\mathcal{G}$.
Then we can get the refinement of Bowditch's theorem.

\begin{thmalpha} [Gromov \cite{MR919829}, Delzant \cite{MR1390660}, Bowditch \cite{Bowditch08}, Theorem \ref{thm:rational_length}] \label{thmalpha:b}
Let $G$ be a group acting simplicially on a $\delta$-hyperbolic graph $\mathcal{G}$.
\begin{enumerate}
\item \label{enum1:rational_length} If the action of $G$ has the finiteness property, then $\operatorname{Spec}(G, \mathcal{G})$ consists of rational numbers.
\item \label{enum2:rational_length} If the action of $G$ has the $\kappa$-finiteness property for some positive integer $\kappa$, then $\operatorname{Spec}(G, \mathcal{G})$ consists of fractions of denominator at most $\kappa$.
\end{enumerate}
\end{thmalpha}

\begin{rem}[Curve graph]
After Bowditch \cite{Bowditch08}, Shackleton \cite{MR2970053} and Webb \cite{MR3378831} improved the common denominator of the length spectrum of a mapping class group. In fact we can immediately apply Theorem \ref{thmalpha:b} to Webb's work \cite[Theorem 6.2]{MR3378831}.

Note Hensel--Przytycki--Webb \cite{HenselPrzytyckiWebb15} showed every curve graph is $17$-hyperbolic.
As a result, if $\xi(S) \geq 2$, then the asymptotic translation length of a pseudo-Anosov on $\mathcal{C}(S)$ is a rational number whose denominator is at most $(820 \cdot 2^{2,017,200(\xi(S)+9)})^{\xi(S)}$.
\end{rem}

The second step in the proof of the main theorem is showing that the right-angled Artin group actions on the extension graphs satisfy the $\kappa$-finiteness property for some $\kappa$. We remark that in the general case $\kappa$ depends on the choice of an element, and we get the first part of the main theorem (See Theorem \ref{thm:generalcasefiniteness}). On the other hand, when the graph has girth at least 6, we can show that $\kappa$ can be made into a uniform constant over the entire right-angled Artin group and get the second part of the main theorem (See Theorem \ref{thm:RAAG_finiteness}.) The $\kappa$-finiteness induces the following theorem from which the main theorem follows.

\begin{thmalpha}[Corollary \ref{cor:RAAG_cyclic_permutation}]
Let $\Gamma$ be a finite connected simplicial graph of girth at least 6.
If $N$ is the maximum degree of $\Gamma$, then every loxodromic of $A(\Gamma)$ permutes cyclically at most $N$ geodesics on $\Gamma^e$.
\end{thmalpha}

Here the \emph{degree} of a vertex is the number of edges incident to the vertex, and the \emph{maximum degree} of a graph is the maximum of the degrees of all vertices of the graph.

In Section \ref{sec:examples}, we introduce some applications and examples induced from the $\kappa$-finiteness property of a right-angled Artin group.
In Section \ref{subsec:treegraphs} and Section \ref{subsec:cyclegraphs}, we study the possible asymptotic translation lengths in the case when the defining graph for the right-angled Artin group is either a tree or a cycle.
We simply write $\operatorname{Spec}(A(\Gamma))$ for the length spectrum $\operatorname{Spec}(A(\Gamma), \Gamma^e)$.

\begin{thmalpha}[Proposition \ref{thm:translation-tree} and \ref{prop:spectrum_cycle}]
For each finite simplicial graph $\Gamma$, the following hold.
\begin{enumerate}
\item If $\Gamma$ is a tree, then $\operatorname{Spec}(A(\Gamma))$ is a set of even integers.
\item If $\Gamma$ is a cycle of even length more than $5$, then $\operatorname{Spec}(A(\Gamma))$ is a set of integers.
\item If $\Gamma$ is a cycle of odd length more than $5$, then $\operatorname{Spec}(A(\Gamma))$ is a set of rational numbers of denominator $2$, furthermore, it contains a non-integer value.
\end{enumerate}
\end{thmalpha}

We also consider the realization problem. Namely, which rational numbers can be realized as asymptotic translation lengths of the right-angled Artin group action on the extension graph? In Section \ref{subsec:arbitrarydenominator}, we show that the denominators of the asymptotic translation lengths of the right-angled Artin group action on the extension graph can be arbitrary. 

\begin{thmalpha}[Proposition \ref{prop:arbitrary_integer}]
For every positive integer $k$ more than $2$, there exists a finite simplicial graph $\Gamma$ such that $A(\Gamma)$ contains an element of asymptotic translation length $3 + (1/k)$.
\end{thmalpha}

Finally, we obtain a uniform bound of minimum positive asymptotic translation length.

\begin{thmalpha}[Corollary \ref{cor:length_two}]
For every finite connected simplicial graph $\Gamma$ of diameter at least $3$, the minimum positive asymptotic translation length for $A(\Gamma)$ is at most $2$.
\end{thmalpha}

\subsection{Acknowledgements}
We thank Anthony Genevois, Takuya Katayama, Sang-hyun Kim, Eiko Kin, Ki-hyoung Ko, and Thomas Koberda for helpful discussions. We also appreciate the organizers of the 16th East Asian Conference on Geometric Topology since many discussions beneficial for us to develop this research happened during the conference. We would like to give special thanks to Anthony Genevois for pointing out an error in the earlier draft. The first author was partially supported by the National Research Foundation of Korea(NRF) grant funded by the Korea government(MSIT) (No. 2020R1C1C1A01006912), the second author was supported by the National Research Foundation of Korea(NRF) grant No. 2021R1C1C200593811 from the Korea government(MSIT), and the last author was supported by Samsung Science and Technology Foundation under Project Number SSTF-BA1702-01.

%
%
\section{Preliminary} \label{sec:prelim}

We give a brief review on basic materials used in this paper regarding the right-angled Artin groups and the extension graphs. Let us recall the assumption that $\Gamma$ is connected. Some well-known facts seem to lack proofs in the literature. We provide proofs to those facts for the sake of completeness.

\subsection{Right-angled Artin group}

Suppose $\Gamma = (V(\Gamma), E(\Gamma))$ is a finite connected simplicial graph. 
The \emph{right-angled Artin group} of $\Gamma$, denoted by $A(\Gamma)$, is the group presented by
$$A(\Gamma) := \langle v \in V(\Gamma) \mid [u, v] = 1 ~\text{for}~\text{all}~\{u, v\} \in E(\Gamma) \rangle.$$
One say $\Gamma$ is the \emph{defining graph} of $A(\Gamma)$.
The defining graph $\Gamma$ is frequently considered as a metric space with respect to the \emph{edge metric} $d_\Gamma$ on $V(\Gamma)$ that measures the smallest number of edges between two vertices.
For each vertex $v \in \Gamma$, the \emph{star of $v$}, denoted by $\operatorname{st}_\Gamma(v)$, is the subgraph induced with $\{u \in V(\Gamma) \mid d_\Gamma(u, v) = 1 \}$.

The \emph{support} of a reduced word $w$, denoted by $\operatorname{supp}(w)$, is the set of vertices composing $w$.
By Hermiller--Meier \cite{MR1314099}, all reduced words representing the same element have identical word length and the same supports.
In other words, for every element $g \in A(\Gamma)$, one can define $\operatorname{supp}(g)$, the \emph{support of $g$}, as a support of a reduced word representing $g$.
The \emph{word length} of $g$ is a length of a reduced word representing $g$.

A reduced word $w$ of $A(\Gamma)$ is called a \emph{star word} if $\operatorname{supp}(w) \subseteq \operatorname{st}_\Gamma(v)$ for some $v \in V(\Gamma)$.
Note every word can be decomposed into a product of star words.
The \emph{star length} of $g$, written by $\| g \|_{\rm st}$, is the smallest number of star words whose product represents $g$.
The next lemma indicates every element admits a product of star words that optimizes word length and star length simultaneously.

\begin{lem}[Lemma 20(2) \cite{kim2014geometry}] \label{lem:st}
For every element $g \in A(\Gamma)$, there exist star words $w_1, \dots, w_n$ such that $g = w_n \dots w_1$ with $\lvert g \rvert = \lvert w_n \rvert + \dots + \lvert w_1 \rvert$ and $\left\| g \right\|_{\operatorname{st}} = n$.
\end{lem}

For an element $g \in A(\Gamma)$, a \emph{star decomposition} of $g$ is a product of star words, $g = w_n \dots w_1$, satisfying $\lvert g \rvert = \sum_{i = 1}^n\lvert w_i  \rvert$ and $\| g \|_{\rm st} = n$.
A \emph{syllable} is a nonzero power of a vertex generator of $A(\Gamma)$, that is, $v^\ell$ for some $v \in V(\Gamma)$ and $\ell \neq 0$.
Note each element in $A(\Gamma)$ can be written as a product of syllables.
For each $g \in A(\Gamma)$, the minimum length of a product of syllables representing $g$ is called the \emph{syllable length} of $g$, denoted by $\| g \|_{\rm syl}$, and let us call a shortest product of syllables representing $g$ a \emph{syllable decomposition} of $g$.

\subsection{Extension graph}

For two elements $g, h \in A(\Gamma)$, let $g^h$ denote $h^{-1}gh$.
Let us define a simplicial graph $\Gamma^e$: the vertex set of $\Gamma^e$ consists of $v^g$ for all $v \in V(\Gamma)$ and $g \in A(\Gamma)$, and $u^h$ and $v^g$ are joined by an edge whenever $[u^h, v^g] = 1$.
We call $\Gamma^e$ the \emph{extension graph} of $\Gamma$.
The defining graph $\Gamma$ can be considered as a subgraph of $\Gamma^e$ via the inclusion $\Gamma \hookrightarrow \Gamma^e$ defined by $v \mapsto v$.
From this point of view, $\Gamma^g$ can be seen as the image of $\Gamma$ in the extension graph.
For every $x \in \Gamma^e$, one write the \emph{star of $x$ in $\Gamma^e$} by $\operatorname{st}_{\Gamma^e}(x)$.

\begin{rem}[Join]
The \emph{(graph) join} of two simplicial graphs $\Delta$ and $\Delta'$, denoted by $\Delta \oplus \Delta'$, is obtained from the disjoint union $\Delta \sqcup \Delta'$ by adding all edges joining pairs of vertices in $V(\Delta) \times V(\Delta')$.
As a binary operation of simplicial graphs, join is commutative and associative.
Furthermore, join and direct sum are completely compatible in the collection of right-angled Artin groups.
Precisely, the homomorphism $A(\Delta) \oplus A(\Delta') \to A(\Delta \oplus \Delta')$, defined by $(x, y) \mapsto xy$, is an isomorphism; on the other hand, by Centralizer theorem \cite{MR1023285} and Isomorphism theorem \cite{MR1356145}, every right-angled Artin group that can be expressed as a direct sum of two nontrivial subgroups has a nontrivial join as a defining graph.
\end{rem}
\begin{rem}
Let us quickly sketch the proof of the last statement of the above remark.
Suppose $A(\Delta) = A \oplus B$ for some $\Delta$ and nontrivial subgroups $A, B < A(\Delta)$.
Write $C(A) = C_{A(\Delta)}(A)$ and $Z(A)$ for the centralizer and the center of $A$, respectively.
Similarly write $C(B)$ and $Z(B)$ for $B$.
Let us define $\operatorname{lk}_\Delta(A) := (V(\Delta) \cap C(A)) - A$ and $\operatorname{lk}_\Delta(B) := (V(\Delta) \cap C(B)) - B$.
By Centralizer theorem, we get $C(A) = Z(A) \oplus \langle \operatorname{lk}_\Delta(A) \rangle$ and $C(B) = Z(B) \oplus \langle \operatorname{lk}_\Delta(B) \rangle$.
Since $A \cap B = \{1\}$, $A \leq C(B)$ and $B \leq C(A)$, we have $A \leq \langle \operatorname{lk}_\Delta(B) \rangle$ and $B \leq \langle \operatorname{lk}_\Delta(A) \rangle$. This induces $A = \langle \operatorname{lk}_\Delta(B) \rangle$ and $B = \langle \operatorname{lk}_\Delta(A) \rangle$. If $\Delta_A$ (\emph{resp.}, $\Delta_B$) is the induced subgraph with vertex set $\operatorname{lk}_\Delta(B)$ (\emph{resp.}, $\operatorname{lk}_\Delta(A)$), then $A \cong A(\Delta_A)$ and $B \cong A(\Delta_B)$. Therefore, $A(\Delta)$ is isomorphic to $A( \Delta_A \oplus \Delta_B)$, and by Isomorphism theorem, $\Delta \cong \Delta_A \oplus \Delta_B$.
\end{rem}

If $\Gamma = \Gamma_1 \oplus \Gamma_2$ for some nonempty graphs $\Gamma_1$ and $\Gamma_2$, then $\Gamma^e = \Gamma_1^e \oplus \Gamma_2^e$, so $\Gamma^e$ is bounded.
In fact, Kim--Koberda \cite[Lemma 3.5(5)]{kim2013embedability} showed $\Gamma^e$ is bounded if and only if either $\Gamma$ is a single vertex or can be decomposed into a join of nonempty subgraphs.
The supposition in the next lemma is equivalent that $\Gamma^e$ is unbounded.

\begin{lem} \label{lem:star}
  Suppose $\Gamma$ contains more than one vertex and cannot be decomposed into a join of two nonempty subgraphs.
  \begin{enumerate}
    \item \label{enum:star1} For each $v \in V(\Gamma)$ and $\ell \neq 0$, one has $\operatorname{st}_\Gamma(v) = \Gamma \cap \Gamma^{v^\ell}$. 
    \item \cite[Lemma 3.5(6)]{kim2013embedability} \label{enum:star2} For every vertex $x \in \Gamma^e$, the complement of $\operatorname{st}_{\Gamma^e}(x)$ is disconnected.
  \end{enumerate}
\end{lem}

A doubling is one of the fundamental tools to study $\Gamma^e$, developed by Kim--Koberda \cite{kim2013embedability, kim2014geometry}.
Precisely, for a subgraph $A$ of $\Gamma^e$, a \emph{doubling} of $A$ along a vertex $v$ is the union of $A$ and $A^{v^\ell}$ for some nonzero integer $\ell$.
Note there exists an infinite sequence $\Gamma = \Gamma_0 \subset \Gamma_1 \subset \Gamma_2 \subset \cdots$ such that each $\Gamma_{i+1}$ is a doubling of $\Gamma_i$ and $\Gamma^e  = \bigcup_{i = 0}^\infty \Gamma_i$.
As following the proof of \cite[Lemma 3.5(6)]{kim2013embedability}, we prove the following.

\begin{lem} \label{lem:sylsep}
Let $g = s_n \dots s_1$ be a syllable decomposition.
For all $i \in \{1, \dots, n\}$, if $v_i$ is the vertex supporting $s_i$ and $z_i$ denotes $v_i^{s_i \dots s_1}$, then $\Gamma - \operatorname{st}_{\Gamma^e}(z_i)$ and $\Gamma^g - \operatorname{st}_{\Gamma^e}(z_i)$ are separated by $\operatorname{st}_{\Gamma^e}(z_i)$.
\end{lem}

\begin{proof}
Let $A$ denote $\Gamma \cup \bigcup_{k = 1}^n \Gamma^{s_k \dots s_1}$.
Lemma \ref{lem:star}\eqref{enum:star1} implies $\Gamma \cap \Gamma^{s_1} = \operatorname{st}_\Gamma(v_1)$ and $\Gamma^{s_k \dots s_1} \cap \Gamma^{s_{k+1} \dots s_1} = \operatorname{st}_{\Gamma}(v_{k+1})^{s_{k+1} \dots s_1}$ for each $k$.
Because $\Gamma$ is connected, $A$ is connected as well.

Choose $x = u^{s_j \dots s_1}$ and $y = v^{s_\ell \dots s_1}$ for some $u,v \in V(\Gamma)$ and $0 \leq j < i \leq \ell \leq n$. (Every empty word in this proof is considered as the identity, for example, $j = 0 \Rightarrow x = u$.)
We claim if $[x, y] = 1$, then either $[x, z_i] = 1$ or $[y, z_i] = 1$.
Because $[u, v^{s_\ell \dots s_{j+1}}] = 1$, there exists a syllable decomposition $s_\ell \dots s_{j+1} = s_\ell' \dots s_{j+1}'$ such that $[u, s_{j_0}' \dots s_{j+1}'] = 1$ and $[v, s_\ell' \dots s_{j_0+1}'] = 1$ for some $j_0 \in \{ j, \dots, \ell \}$.

Since $j < i \leq \ell$, the syllable $s_i$ belongs to either $\{ s_{j+1}', \dots, s_{j_0}' \}$ or $\{s_{j_0 +1}', \dots, s_\ell'\}$.
If $s_i = s_{j_1}'$ for some $j+1 \leq j_1 \leq j_0$, then $u$ and $v_i$ commute by Centralizer theorem \cite{MR1023285}.
Because $z_i = v_i^{(s_{j_1}' \dots s_{j+1}') (s_j \dots s_1)}$ by Lemma \ref{lem:z_invariant}, we have $$[x, z_i] = [u^{s_j \dots s_1}, v_i^{s_i \dots s_1}] = [u^{(s_{j_1}' \dots s_{j+1}')(s_j \dots s_1)}, v_i^{(s_{j_1}' \dots s_{j+1}') (s_j \dots s_1)}] = 1$$
Similarly, if $s_i \in \{s_{j_0 +1}', \dots, s_\ell'\}$, then $y$ and $z_i$ commute.
So the claim holds.

By the above claim, $B := (\Gamma \cup (\bigcup_{k = 1}^{i - 1}\Gamma^{s_k \dots s_1})) - \operatorname{st}_{\Gamma^e}(z_i)$ and $C := (\bigcup_{k = i}^n \Gamma^{s_k \dots s_1}) - \operatorname{st}_{\Gamma^e}(z_i)$ are disjoint, and furthermore, they are not joined by an edge.
Note $A - \operatorname{st}_{\Gamma^e}(z_i) = B \sqcup C$.
So we have $A - \operatorname{st}_{\Gamma^e}(z_i)$ is disconnected.
More precisely, $\operatorname{st}_{\Gamma^e}(z_i)$ separates $\Gamma - \operatorname{st}_{\Gamma^e}(z_i)$ from $\Gamma^g - \operatorname{st}_{\Gamma^e}(z_i)$ in $A$.

Because $\Gamma \subset A$, there exists a sequence $A = A_0 \subset A_1 \subset A_2 \subset \dots$ such that $\Gamma^e = \bigcup_{k = 0}^\infty A_k$ and $A_{k+1}$ is a doubling of $A_k$ for each $k$.
We need only to show that for each $k$, $\operatorname{st}_{\Gamma^e}(z_i)$ separates $\Gamma - \operatorname{st}_{\Gamma^e}(z_i)$ from $\Gamma^g - \operatorname{st}_{\Gamma^e}(z_i)$ in $A_k$.
Assume that $\operatorname{st}_{\Gamma^e}(z_i)$ separates $\Gamma - \operatorname{st}_{\Gamma^e}(z_i)$ from $\Gamma^g - \operatorname{st}_{\Gamma^e}(z_i)$ in $A_{k - 1}$.

Suppose $A_k = A_{k - 1} \cup A_{k - 1}^{(u')^m}$ for some $u'$ and $m \in \mathbb{Z} \setminus \{ 0 \}$.
Then there exists a simplicial projection $p: A_k \to A_{k-1}$ defined by $p(x) = \begin{cases} x & \text{if}~x \in A_{k-1}, \\ x^{(u')^{-m}}, & \text{otherwise}. \end{cases}$
Choose a path subgraph $P \subset A_k$ joining $\Gamma - \operatorname{st}_{\Gamma^e}(z_i)$ and $\Gamma^g - \operatorname{st}_{\Gamma^e}(z_i)$.

By the induction hypothesis, $p(P)$ intersects $\operatorname{st}_{\Gamma^e}(z_i)$.
If $u' \in \operatorname{st}_{\Gamma^e}(z_i)$, then $P$ intersects $\operatorname{st}_{\Gamma^e}(z_i)$ because $(u')^m$ preserves $\operatorname{st}_{\Gamma^e}(z_i)$.
Suppose $u'$ does not commute with $z_i$.
Then $u'$ is contained in a component of $A_{k-1} - \operatorname{st}_{\Gamma^e}(z_i)$.

If $u'$ and $\Gamma - \operatorname{st}_{\Gamma^e}(z_i)$ are in different components of $A_{k-1} - \operatorname{st}_{\Gamma^e}(z_i)$, then the component of $P \cap A_{k - 1}$ starting from $\Gamma - \operatorname{st}_{\Gamma^e}(z_i)$ connects either $\operatorname{st}_{\Gamma^e}(z_i)$ or $\Gamma^g - \operatorname{st}_{\Gamma^e}(z_i)$ so that this component intersects $\operatorname{st}_{\Gamma^e}(z_i)$.
So $P$ intersects $\operatorname{st}_{\Gamma^e}(z_i)$.
Similarly, if $u'$ and $\Gamma^g - \operatorname{st}_{\Gamma^e}(z_i)$ are in different components of $A_{k-1} - \operatorname{st}_{\Gamma^e}(z_i)$, then $P$ intersects $\operatorname{st}_{\Gamma^e}(z_i)$.

Hence, every path joining $\Gamma - \operatorname{st}_{\Gamma^e}(z_i)$ and $\Gamma^g - \operatorname{st}_{\Gamma^e}(z_i)$ in $A_k$ intersects $\operatorname{st}_{\Gamma^e}(z_i)$.
This implies for all $k$, $\operatorname{st}_{\Gamma^e}(z_i)$ separates $\Gamma - \operatorname{st}_{\Gamma^e}(z_i)$ from $\Gamma^g - \operatorname{st}_{\Gamma^e}(z_i)$ in $A_k$.
Therefore, $\operatorname{st}_{\Gamma^e}(z_i)$ separates $\Gamma - \operatorname{st}_{\Gamma^e}(z_i)$ from $\Gamma^g - \operatorname{st}_{\Gamma^e}(z_i)$ in $\Gamma^e$.
\end{proof}

\begin{lem}[Lemma 3.5(7) \cite{kim2013embedability}]
The extension graph $\Gamma^e$ is a quasi-tree. In particular, it is $\delta$-hyperbolic.
\end{lem}

\subsubsection{The simplicial action of $A(\Gamma)$ on $\Gamma^e$} 

We define a \emph{right action} of $A(\Gamma)$ on $\Gamma^e$ by $$g: x \mapsto x^g$$ for all $g \in A(\Gamma)$ and $x \in \Gamma^e$.
This definition can be extended as a simplicial action of $A(\Gamma)$ on $\Gamma^e$.
If $d_{\Gamma^e}$ is the edge metric on $\Gamma^e$, then the action is an isometric action, so each element of $A(\Gamma)$ has translation length on $\Gamma^e$.
An element $g \in A(\Gamma)$ is said to be \emph{elliptic} if some orbit of $\langle g \rangle$ is bounded.
On the other hand, an element $g$ of $A(\Gamma)$ is \emph{loxodromic} if the translation length, denoted by $\tau(g)$, is positive.

Kim--Koberda \cite[Theorem 30]{kim2014geometry} proved the action of $A(\Gamma)$ on $\Gamma^e$ is acylindrical.
So each element of $A(\Gamma)$ is elliptic or loxodromic, which was shown by Bowditch \cite[Lemma 2.2]{Bowditch08}.
The following proposition provides some conditions which tell us when an element of $A(\Gamma)$ is elliptic.

\begin{prop}[Lemma 34, Theorem 35 \cite{kim2014geometry}] \label{prop:ellip}
Let $g$ be a cyclically reduced element of $A(\Gamma)$.
Then the following are equivalent.
\begin{enumerate}
\item $g$ is elliptic.
\item $\operatorname{supp}(g)$ is contained in a join of $\Gamma$.
\item The sequence $(\| g^n \|_{\rm st})_{n \in \mathbb{N}}$ is bounded.
\end{enumerate}
\end{prop}

\subsection{Syllable decomposition}

In this section, we deal with the algorithmic property of syllable decomposition of an element.

\subsubsection{Elementary moves on syllable length}

After Newman's insight about moves on the set of words of a free group \cite{MR7372}, Hermiller--Meier \cite{MR1314099} introduced Newman's idea to graph products of groups, which handles all right-angled Artin groups.
For each $g \in A(\Gamma)$, there are two elementary moves on the set of words representing $g$: one move is ``switching a consecutive commuting alphabets'' (transposition); and the other is ``removing a consecutive inverse pair'' (cancellation).
Hermiller--Meier showed for every word $w$ and a reduced word $w'$, if $w$ and $w'$ represent $g$, then $w$ can be reformed to $w'$ by a combination of finitely many transpositions and cancellations.
Notice the length of a word decreases if and only if a cancellation acts on it.
So, if we start from a reduced word, there is no possible cancellation.
That is, for two reduced words that represent $g$, one can be reformed to the other by finitely many transpositions.

Hermiller--Meier's system also works on syllable decompositions and syllable length.
In detail, for an element $g \in A(\Gamma)$, let us consider the set of all products of (finitely many) syllables that represent $g$. Define a length of each product by the number of syllables. Then every word that represents $g$ is contained in this set since each alphabet can be seen as a syllable.
Then there are three elementary moves ``switching two consecutive commuting syllables'' (syllable transposition), ``summing up two consecutive syllables of a same support'' (syllable summation) and ``removing a zero power'' (syllable cancellation) such that a product of syllables can be reformed into any syllable decomposition by a composition of finitely many moves.
Similar to the above, the number of syllables decreases if and only if a syllable summation or a syllable cancellation acts on it. So, for two syllable decompositions that represent $g$, one can be reformed to the other by finitely many syllable transpositions.

\begin{lem}[Hermiller--Meier \cite{MR1314099}] \label{lem:syl_express}
For every $g \in A(\Gamma)$ and two syllable decompositions that represent $g$, one can be reformed to the other by a composition of syllable transpositions.
\end{lem}

\subsubsection{Some characteristic points of an isometry}

Let $s_n \dots s_1$ be a syllable decomposition of an element $g$.
For every $1 \leq i \leq j \leq n$, we call the subword $s_j \dots s_i$ a \emph{syllable subword} of $s_n \dots s_1$.
We say a syllable subword $s_j \dots s_i$ is \emph{rightmost} if $i = 1$.

For distinct syllable decompositions $s_n \dots s_1$ and $s_{\sigma(n)} \dots s_{\sigma(1)}$ that represent $g$, some rightmost syllable subword of $s_n \dots s_1$ may not be reformed to any rightmost syllable subword of $s_{\sigma(n)} \dots s_{\sigma(1)}$.
For example, assume $s_j$ and $s_{j+1}$ commute for some $1 \leq j < n$ and $\tau$ is the transposition of $j$ and $j+1$.
Let us compare $s_j \dots s_1$ with $s_{\tau(i)} \dots s_{\tau(1)}$ for each $i$.
If $i \neq j$, then $\| s_j \dots s_1 \|_{\rm syl} \neq \| s_{\tau(i)} \dots s_{\tau(1)} \|_{\rm syl}$, so they cannot represent a same element.
Otherwise, the difference $s_j \dots s_1 (s_{\tau(j)} \dots s_{\tau(1)})^{-1} = s_j s_{j+1}^{-1}$ does not represent the identity because $s_j$ and $s_{j+1}$ have different supports.
Therefore, $s_j \dots s_1$ cannot be represented by any syllable subwords of $s_{\tau(n)} \dots s_{\tau(1)}$.

\begin{lem} \label{lem:additional_lem}
Let $g \in A(\Gamma)$ be given, and let $g = s_n \dots s_1 = s_{\sigma(n)} \dots s_{\sigma(1)}$ denote two syllable decompositions.
For $i_0, j_0 \in \{1, \dots, n\}$, if each of $s_{i_0} \dots s_1$ and $s_{\sigma(j_0)} \dots s_{\sigma(1)}$ cannot be reformed to a syllable subword of the other, then there are nontrivial commuting syllable subwords $w_1, w_2$ of $g$ such that the following hold:
\begin{enumerate}
\item $s_{i_0} \dots s_1 = w_1 w_0$ and $s_{\sigma(j_0)} \dots s_{\sigma(1)} = w_2 w_0$ for some rightmost syllable subword $w_0$ of a syllable decomposition of $g$, and
\item the join $\operatorname{supp}(w_1) \oplus \operatorname{supp}(w_2)$ is embedded in $\Gamma$.
\end{enumerate}
In addition, if $i_0 = \sigma(j_0)$, then $s_{i_0}$ commutes with $w_1$ and $w_2$.
\end{lem}

\begin{proof}
Before we start, let us define some notations only used for this proof.
For all $\ell, \ell' \in \{1, \dots, n\}$, write $\tau_{\ell, \ell'}$ by a transposition for commuting syllables $s_{\ell}$ and $s_{\ell'}$.
Write $I$ as the intersection of $\{ 1, \dots, i_0 \}$ and $\{ \sigma(1), \dots, \sigma(j_0) \}$.
By the hypothesis, $\{ 1, \dots, i_0 \} \setminus I$ and $\{ \sigma(1), \dots, \sigma(j_0) \} \setminus I$ are nonempty.

If $i \not\in I$ and $i+1 \in I$ for some $i$, then a product of transpositions that represents $\sigma$ contains $\tau_{i, i+1}$, that is, $s_i$ and $s_{i+1}$ commute, so one can move $s_i$ to the left of $s_{i+1}$.
After this movement, we need to rewrite $s_{i_0} \dots s_1$ by $s_{i_0} \dots (s_i s_{i+1}) \dots s_1$ and also redefine $\sigma$ by $\sigma \cdot \tau_{i, i+1}$.
Let us repeat such steps until all syllables indexed by $I$ locate on the rightmost of $s_{i_0} \dots s_1$.
Let $w_0$ denote the rightmost syllable subword consisting of all syllables indexed by $I$, and let $w_1$ be a syllable subword that consists of the remaining syllables.
Then we obtain an equation $s_{i_0} \dots s_1 = w_1 w_0$.

Similarly, whenever $\sigma(j) \not\in I$ and $\sigma(j+1) \in I$, one can move $s_{\sigma(j)}$ to the left of $s_{\sigma(j+1)}$, rewrite $s_{\sigma(j_0)} \dots s_{\sigma(1)}$ by $s_{\sigma(j_0)} \dots (s_{\sigma(j)} s_{\sigma(j+1)}) \dots s_{\sigma(1)}$ and redefine $\sigma$ by $\tau_{\sigma(j) \sigma(j+1)} \cdot \sigma$.
After iterating such steps, all syllables indexed by $I$ will be located on the rightmost of $s_{\sigma(j_0)} \dots s_{\sigma(1)}$ that also form the syllable subword $w_0$ defined on the above paragraph.
Let $w_2$ be a syllable subword consisting of the remaining syllables.
Then one write $s_{\sigma(j_0)} \dots s_{\sigma(1)} = w_2 w_0$.

By the hypothesis of the statement, $\operatorname{supp}(w_1)$ and $\operatorname{supp}(w_2)$ are nonempty.
Note every syllable in $w_1$ has index at most $i_0$ while an index of a syllable in $w_2$ is larger than $i_0$.
So for all $i \in \{1, \dots, i_0\} \setminus I$ and $j \in \{\sigma(1), \dots, \sigma(j_0)\} \setminus I$, a product of transpositions representing $\sigma$ contains $\tau_{i, j}$.
This implies $w_1$ and $w_2$ commute, therefore, the join of $\operatorname{supp}(w_1)$ and $\operatorname{supp}(w_2)$ is embedded into $\Gamma$.

If $i_0 = \sigma(j_0)$, then a product of transpositions representing $\sigma$ contains syllable transpositions $\tau_{i, i_0}$ for all $i \in ( \{ 1, \dots, i_0 \} \cup \{ \sigma(1), \dots, \sigma(j_0) \} ) \setminus I$.
So $s_{i_0}$ commutes with $w_1$ and $w_2$.
Hence, $w_2 w_1^{-1}$ is a star word that commutes with $s_{i_0}$.
\end{proof}

The additional statement of Lemma \ref{lem:additional_lem} deduces the next lemma.

\begin{lem} \label{lem:z_invariant}
Let $g = s_n \dots s_1 = s_{\sigma(n)} \dots s_{\sigma(1)}$ be two syllable decompositions of $g \in A(\Gamma)$.
For each $i$, let $v_i$ denote the vertex supporting $s_i$, and write $z_i := v_i^{s_i \dots s_1}$.
Then we have $$z_{\sigma(i)} = v_{\sigma(i)}^{s_{\sigma(i)} \dots s_{\sigma(1)}}$$ for every $i$.
\end{lem}

\begin{proof}
By Lemma \ref{lem:additional_lem}, one can write $z_{\sigma(i)} = v_{\sigma(i)}^{w_1 w_0}$ and $v_{\sigma(i)}^{s_{\sigma(i)} \dots s_{\sigma(1)}} = v_{\sigma(i)}^{w_2w_0}$ for some $w_0, w_1, w_2$ such that $w_1$ and $w_2$ are star words of $v_{\sigma(i)}$.
So we have $$ z_{\sigma(i)} =  v_{\sigma(i)}^{w_1 w_0} = v_{\sigma(i)}^{w_0} = v_{\sigma(i)}^{w_2 w_0} = v_{\sigma(i)}^{s_{\sigma(i)} \dots s_{\sigma(1)}}.$$
Hence, the statement holds.
\end{proof}

\subsubsection{Cyclically syllable-reduced element}

We say an element $g \in A(\Gamma)$ is \emph{cyclically syllable-reduced} if it has the minimum syllable length in the conjugacy class of $g$, that is, $\| g \|_{\rm syl} = \min_{h \in A(\Gamma)} \|g^h\|_{\rm syl}$.
Every cyclically syllable-reduced element is cyclically reduced, but the converse does not hold.
For example, if two vertices $u$ and $v$ do not commute, then the word $uvu$ is cyclically reduced but not cyclically syllable-reduced. 
This is because $u^2v$ is conjugate to $uvu$ and has syllable length $2$ while the syllable length of $uvu$ is $3$.

\begin{lem} \label{lem:cyclic_syllable}
If $g \in A(\Gamma)$ is a cyclically syllable-reduced element that is not represented by a star word, then we have $\| g^m \|_{\rm syl} = \lvert m \rvert \cdot \| g \|_{\rm syl}$ for every $m \in \mathbb{Z}$.
\end{lem}

\begin{proof}
For proof by contradiction, suppose $\| g^m \|_{\rm syl} < \lvert m \rvert \cdot \| g \|_{\rm syl}$ for some $m > 0$.
Let $g = s_n \dots s_1$ be a syllable decomposition of $g$.
By supposition, $(s_n \dots s_1)^m$ is not a syllable decomposition of $g^m$.
That is, if we relabel $g^m = (s_n \dots s_1) \dots (s_n \dots s_1)$ as $s_{mn} s_{mn - 1} \dots s_1$, then there exist $1 \leq a < b \leq mn$ such that $\operatorname{supp}(s_a) = \operatorname{supp}(s_b)$ and $[s_a, s_i] = 1$ for all $a \leq i \leq b$.

If $\lvert a - b \rvert \geq n - 1$, then $s_a$ commutes all copies of syllables of $g$ so that $\operatorname{supp}(g) \subseteq \operatorname{st}_\Gamma(v)$ where $v$ is the vertex supporting $s_a$.
It gives $g$ is a star word, which is a contradiction.
If $\lvert a - b \rvert < n - 1$, then choose $a', b'$ satisfying $a' \leq a < b \leq b'$ and $b' - a' = n - 1$.
Then the word $s_{b'} \dots s_{a'}$ is conjugate to $g$, but its syllable length is smaller than $g$.
This also gives a contradiction against the condition of $g$.
So such $s_a$ and $s_b$ cannot exist.
Therefore, the supposition is false, that is, the equality $\| g^m \|_{\rm syl} = \lvert m \rvert \cdot \| g \|_{\rm syl}$ holds for every $m \in \mathbb{Z}$.
\end{proof}

\subsubsection{Cyclically syllable-reduced loxodromic}

If $g  = s_n \dots s_1$ is a syllable decomposition of a cyclically syllable-reduced loxodromic, then $(s_n \dots s_1)^m$ is a syllable decomposition of $g^m$ for every positive integer $m$ by Lemma \ref{lem:cyclic_syllable}.
For each $j \in \{1, \dots, n\}$ and $\ell \in \{1, \dots, m - 1\}$, let $s_{j + \ell n}$ denote a copy of $s_j$.
Then $g^m$ can be written as $s_{mn} \dots s_1$.

This implies every elliptic subword of (a reduced word) representing $g$ is also a subword of a word representing a positive power of $g$.
But for sufficiently large $m$, some elliptic subword of $g^m$ might not be representable as a subword of $g$.
This is because $g^m$ may admit a syllable decomposition that is not a concatenation of syllable decompositions of $g$.
Nonetheless, a power of $g$ has a restriction to form a syllable decomposition as follow.

\begin{lem} \label{lem:syl_perm}
Let $g = s_n \dots s_1$ be a syllable decomposition of a cyclically syllable-reduced loxodromic, and let a positive integer $m$ be given.
Suppose $g^m = s_{mn} \dots s_1$ is a concatenation of $m$ copies of $s_n \dots s_1$ and $s_{\sigma(mn)} \dots s_{\sigma(1)}$ is another syllable decomposition of $g^m$.
Then we have $$\lvert i - \sigma^{-1}(i) \rvert \leq n \lvert V(\Gamma) \rvert$$ for all $i \in \{ 1, \dots, mn \}$.
\end{lem}

\begin{proof}
Consider the complement graph of $\Gamma$, denoted by $\bar{\Gamma}$.
Because $g$ is cyclically reduced and loxodromic, the subgraph induced by $\operatorname{supp}(g)$ is connected in $\bar{\Gamma}$.
Meanwhile, the existence of a loxodromic element implies that $\bar\Gamma$ is connected.

Fix an index $i_0 \in \{ 1, \dots, mn \}$, and let $v$ be the vertex supporting $s_{i_0}$.
Write $I := \{1, \dots, mn\}$, and for $a, b \in \mathbb{R}$, let $(a, b]$ denote the interval $\{ c \mid a < c \leq b \}$.
For each integer $\ell \geq 1$, write $$I_\ell := I \cap (i_0 + (\ell - 1)n, i_0 + \ell n] \cap \{ i \mid d_{\bar\Gamma}(v, v_i) \leq \ell \}$$ and  $I_0 := \{ i_0 \}$.

If $\ell < \ell'$ and $i_0 + \ell'n \leq mn$, then $I_\ell$ and $I_{\ell'}$ are disjoint and $$\{v_j \mid j \in I_\ell\} \subseteq \{v_j \mid j \in I_{\ell'}\}.$$
Because $\operatorname{supp}(g)$ is connected, the induced subgraph of $\{v_i \mid i \in I_\ell\}$ is connected for each $\ell$.
We claim that for each $\ell \geq 1$ and $i \in I_\ell$, there exists $j \in I_{\ell - 1}$ such that $\sigma^{-1}(j) < \sigma^{-1}(i)$.

First, consider the case that $i$ satisfies $v_i \in \{ v_j \mid j \in I_{\ell - 1} \} \cap \{ v_j \mid j \in I_\ell \}$.
If $\sigma^{-1}(j) > \sigma^{-1}(i)$, then by Lemma \ref{lem:syl_express}, there exists a transposition between $s_i$ and $s_j$.
This implies $g^m$ admits a word of $mn$ syllables containing the subword $s_i s_j$ so that $s_{mn} \dots s_{1}$ can be reduced.
That is, we have $\| g \|_{\operatorname{syl}} < mn$, which is a contradiction against Lemma \ref{lem:cyclic_syllable}.
So $\sigma^{-1}(j)$ is less than $\sigma^{-1}(i)$.

The other case is that $v_i$ is not contained in $\{ v_j \mid j \in I_{\ell - 1} \}$.
Then we have $d_{\bar\Gamma}(v, v_i) = \ell$ by definition.
Since the induced subgraph of $\{v_j \mid j \in I_\ell\}$ in $\bar\Gamma$ is connected, there exists $j \in I_{\ell - 1}$ such that $v_j$ and $v_i$ are adjacent in $\bar\Gamma$.
It means $v_i$ and $v_j$ do not commute so that $\sigma^{-1}(j) < \sigma^{-1}(i)$ by Lemma \ref{lem:syl_express}.
So the claim holds.

By the above claim, for every $\ell \geq 1$ and $j \in I_\ell$, there exists a sequence $i_0 = j_0 < j_1 < \dots < j_\ell = j$ such that $\sigma^{-1}(i_0) < \dots < \sigma^{-1}(j_\ell)$.
Note if $\operatorname{diam}(\bar\Gamma) \leq \ell \leq m - i_0/n$, then $I_\ell$ has cardinality $n$ since $\operatorname{supp}(g) \subseteq \bar\Gamma$.
So we have $\sigma^{-1}(i_0) < \sigma^{-1}(j)$ for all $j > i_0 + n \cdot \operatorname{diam}(\bar\Gamma)$.

By the pigeonhole principle, the inequality $\sigma^{-1}(i_0) \leq i_0 + n \cdot \operatorname{diam}(\bar\Gamma)$ holds.
Then we get $\sigma^{-1}(i_0) - i_0 \leq n \cdot \operatorname{diam}(\bar\Gamma)$.
By changing the roles of $\sigma^{-1}(i_0)$ and $i_0$, we can show $i_0 - \sigma^{-1}(i_0) \leq n \cdot \operatorname{diam}(\bar\Gamma)$ in a similar way.
Therefore, we have $\lvert i_0 - \sigma^{-1}(i_0) \rvert \leq n \cdot \operatorname{diam}(\bar\Gamma)$.
\end{proof}

The above lemma indicates all syllable decompositions of powers of $g$ can be constructed by finitely many words.
Furthermore, elliptic subwords of powers of $g$ are finitely many. 

\begin{prop} \label{prop:num_elliptic}
For a cyclically syllable-reduced loxodromic $g$, if $w$ is an elliptic syllable subword of some syllable decomposition representing a positive power of $g$, then we have $\| w \|_{\operatorname{syl}} \leq \| g \|_{\operatorname{syl}}  (2 \lvert V(\Gamma) \rvert + 1)$.
Furthermore, the number of elliptic elements that can be realized by subwords of reduced words representing positive powers of $g$ is finite.
\end{prop}

\begin{proof}
Fix a syllable decomposition $g = s_n \dots s_1$ and $m > 0$.
Let $g^m = s_{mn} \dots s_1$ be the concatenation of $m$ copies of $s_n \dots s_1$.
Suppose $w$ is an elliptic syllable subword of $g^m = s_{\sigma(mn)} \dots s_{\sigma(1)}$ for some $m > 0$.
Set $w := s_{\sigma(i_1)} \dots s_{\sigma(i_0)}$ for some $1 \leq i_0 \leq i_1 \leq mn$.

Assume $\| w \|_{\rm syl} = i_1 - i_0 + 1 > n (2 \lvert V(\Gamma) \rvert + 1)$.
Then we have $i_1 > i_0 + 2n \lvert V(\Gamma) \rvert + n - 1$.
Note the subword $s_{i_0 + n\lvert V(\Gamma) \rvert + n} \dots s_{i_0 + n\lvert V(\Gamma) \rvert + 1}$ of length $n$ is cyclically conjugate to $g$, so this subword is cyclically syllable-reduced loxodromic.
Hence, there exists an index $j$ with $1 \leq j - (i_0 + n \lvert V(\Gamma) \rvert) \leq n$ such that $\operatorname{supp}(s_j) \cup \operatorname{supp}(w)$ does not form a subjoin of $\Gamma$.
This implies either $\sigma^{-1}(j) < i_0$ or $\sigma^{-1}(j) > i_1$.

On the other hand, by Lemma \ref{lem:syl_perm}, we have $\lvert \sigma^{-1}(j) - j \rvert \leq n \lvert V(\Gamma) \rvert$.
So the following inequalities hold:
\begin{align*}
\sigma^{-1}(j) &\leq n \lvert V(\Gamma) \rvert + j \leq i_0 + 2n\lvert V(\Gamma) \rvert + n < i_1 ~\text{and} \\
\sigma^{-1}(j) &\geq j - n \lvert V(\Gamma) \rvert \geq i_0 + 1 > i_0.
\end{align*}
This is a contradiction. 
Therefore, we have $\| w \|_{\rm syl} \leq n (2 \lvert V(\Gamma) \rvert + 1)$.

If $w'$ is an elliptic subword of a reduced word representing $g^m$, then by permuting alphabets, we decompose $g^m = w_4 w_3 w_2 w_1$ as a product of syllable subwords, and we can also write $w' = w_3' w_2'$ such that $[w_2, w_3] = 1$ and $w_i'$ is a subword of $w_i$ for each $i = 1, 2$.
Then the syllable length of $w'$ is also bounded by $n(2 \lvert V(\Gamma) \rvert + 1)$.
Because the syllables of $g^m$ are uniformly bounded powers of generators, the word length of $w'$ is bounded by some number $N$ determined by $g$ and $\Gamma$.
Therefore, the number of elliptic elements that can be realized by subwords of reduced words representing powers of $g$ is finite.
\end{proof}

%
%
\section{Finiteness Property} \label{sec:finiteness}

In this section, we refine Bowditch's theorem \cite[Theorem 1.4]{Bowditch08} by reorganizing his work.
Let $G$ denote a group acting simplicially on a $\delta$-hyperbolic graph $\mathcal{G}$ with the edge metric $d_{\mathcal{G}}$.
An element of $G$ is called a \emph{loxodromic} if its asymptotic translation length with respect to $d_{\mathcal{G}}$ is positive, and an element is called an \emph{elliptic} if it has a bounded orbit.

In contrast to the hyperbolic space $\mathbb{H}^n$, a loxodromic of $\mathcal{G}$ may not preserve a geodesic.
This phenomenon comes from the fact that a geodesic on $\mathcal{G}$ is not convex in common.
We need to adopt another concept weaker than convexity.
A subgraph $\mathcal{H}$ of $\mathcal{G}$ is said to be \emph{weakly convex} if the inclusion $(\mathcal{H}, d_{\mathcal{H}}) \hookrightarrow (\mathcal{G}, d_\mathcal{G})$ is an isometric embedding when $d_{\mathcal{H}}$ is the edge metric on $\mathcal{H}$. Then every geodesic on $\mathcal{G}$ is weakly convex in $\mathcal{G}$.

For a surface and its curve graph, Bowditch \cite{Bowditch08} showed every pseudo-Anosov preserves a weakly convex and locally finite subgraph of the curve graph, which is described as the union of tight geodesics.
For a loxodromic $g$, let a weakly convex locally finite subgraph $\mathcal{A}_g$ be called an \emph{axial subgraph} if $g$ preserves $\mathcal{A}_g$ and the induced action of $\langle g \rangle$ on $\mathcal{A}_g$ is cocompact.

By the \v{S}vac--Milnor lemma (see \cite[Proposition I.8.19]{MR1744486} for instance), $\mathcal{A}_g$ is quasi-isometric to a line, so it has exactly two ends.
That is, some bounded ball of $\mathcal{A}_g$ separates the ends of $\mathcal{A}_g$.
Because $\mathcal{A}_g$ is locally finite, this ball has finitely many vertices.
We call a set of vertices of $\mathcal{A}_g$ an \emph{end-separating set of $\mathcal{A}_g$} if this separates the ends of $\mathcal{A}_g$.
The \emph{width} of $\mathcal{A}_g$ is the minimum of the cardinalities of end-separating sets of $\mathcal{A}_g$.
Bowditch \cite[Theorem 1.1]{Bowditch08} observed that there exists $\kappa$, depending only on the surface, such that every pseudo-Anosov has an axial subgraph of width at most $\kappa$.
We say such a property as the $\kappa$-finiteness property.
The precise definition is as follow.

\begin{defn*}[Finiteness property]
Suppose a group $G$ acts simplicially on a $\delta$-hyperbolic graph $\mathcal{G}$.
\begin{enumerate}
\item The action of $G$ on $\mathcal{G}$ is said to have the \emph{finiteness property} if every loxodromic has an axial subgraph.
\item For an integer $\kappa \geq 1$, the action of $G$ on $\mathcal{G}$ is said to have the \emph{$\kappa$-finiteness property} if every loxodromic of $G$ has an axial subgraph of width at most $\kappa$.
\end{enumerate}
\end{defn*}

Bowditch proved the following, motivated by Delzant \cite{MR1390660}.

\begin{lem}[Lemma 3.4 \cite{Bowditch08}] \label{lem:bowditch_lemma}
If a loxodromic $g$ has an axial subgraph $\mathcal{A}_g$ of width $\kappa$,  then  $g^m$ preserves a geodesic for some $0 < m \leq \kappa^2$. More precisely, $g$ permutes $m$ geodesics lying on $\mathcal{A}_g$.
\end{lem}

From this lemma, we can find an effective cardinality of a collection of geodesics preserved by $g$.
See the following.

\begin{lem} \label{lem:refinement_Bowditch}
If a loxodromic $g$ has an axial subgraph $\mathcal{A}_g$ of width $\kappa$, then $g$ cyclically permutes at most $\kappa$ pairwise disjoint geodesics lying on $\mathcal{A}_g$.
\end{lem}

We use a left action in the proof of the above lemma because the readers may feel familiar.
However, this lemma can be applied to a right action. 

\begin{proof}
Since every permutation can be decomposed into disjoint cycles, for every finite collection of geodesics obtained from Lemma \ref{lem:bowditch_lemma}, there is a subcollection whose geodesics are cyclically permuted by $g$.
Let $\mathcal{L}$ be a finite collection of geodesics in $\mathcal{A}_g$ cyclically permuted by $g$.
If $\mathcal{L}$ has more than $\kappa$ geodesics, then a pair of geodesics of $\mathcal{L}$ share a vertex of an end-separating set of cardinality $\kappa$.
That is, if it is true that $g$ permutes pairwise disjoint geodesics in $\mathcal{A}_g$, then we have $\lvert \mathcal{L} \rvert \leq \kappa$.

So it is enough to show if $\mathcal{L}$ has intersecting geodesics, there exists a smaller collection of geodesics preserved by $g$.
Assume $L_0 \in \mathcal{L}$ intersects another geodesic of $\mathcal{L}$.
Because $g$ cyclically permutes geodesics of $\mathcal{L}$, there exists $1 < m < \lvert \mathcal{L} \rvert$ such that $L_0$ and $g^mL_0$ have an intersection.

If $x$ is an intersection vertex of $L_0$ and $g^mL_0$, then $L_0$ contains $x$ and $g^{-m}x$.
Let $\gamma$ be the segment of $L_0$ joining $g^{-m}x$ and $x$.
Then $g^{\ell m}\gamma$ lies on $g^{\ell m}L_0$ for each $\ell \in \mathbb{Z}$.
If $L $ is the concatenation of segments $g^{\ell m}\gamma$, then $L$ is preserved by $g^m$.

In fact, $L$ is a geodesic by Lemma \ref{lem:L_is_geodesic}.
So $g$ preserves the collection $\{ L, gL, \dots, g^{m - 1}L \}$ which is smaller than $\mathcal{L}$.
Therefore, a smallest collection preserved by $g$ consists of pairwise disjoint geodesics, and its cardinality is at most $\kappa$.
\end{proof}

In the proof of Lemma \ref{lem:refinement_Bowditch}, we postpone the proof that $L$ is a geodesic.
If $I$ is a segment of $L$ and $I'$ is another geodesic segment such that $I$ and $I'$ share endpoints, we may obtain a line $L'$ from $L$ by substituting $I$ to $I'$.
Then $L'$ is also a geodesic because every segment of $L'$ has the length equal to the distance of endpoints.
Using this method, we can show the following.

\begin{lem} \label{lem:L_is_geodesic}
In the proof of Lemma \ref{lem:refinement_Bowditch}, $L$ is a geodesic.
\end{lem}

\begin{proof}
For a geodesic $L'$ and $x, y \in L'$, let $I(x, y, L')$ denote the segment of $L'$ joining $x$ and $y$.
Because $g^{\lvert \mathcal{L} \rvert}$ preserves $L_0$, the $\langle g^{\lvert\mathcal{L}\rvert} \rangle$-orbit of $x$ is contained in $L_0 \cap (g^m L_0)$ as a subset.
Let $L_1$ be the geodesic obtained from $L_0$ by substituting $I(x, g^{N_1\lvert\mathcal{L}\rvert}x, L_0)$ to $I(x, g^{N_1\lvert\mathcal{L}\rvert}x, g^mL_0)$ for some sufficiently large $N_1$.
Then $L_1$ contains $\gamma \cup (g^m\gamma)$ as a segment since $\gamma = I(g^{-m}x, x, L_0) \subset L_1$ and $g^m\gamma = I(x, g^mx, g^mL_0) \subset L_1.$

Because $N_1$ is sufficiently large, $L_1$ follows $g^m L_0$ for a long time so that $L_1$ contains $g^{m + N_2\lvert\mathcal{L}\rvert}x$ for some large $N_2 < N_1$.
Let $L_2$ be the geodesic obtained by substituting $I(g^mx, g^{m + N_2\lvert\mathcal{L}\rvert}x, g^mL_0) = g^m I(x, x^{N_2\lvert\mathcal{L}\rvert}x, L_0)$ to $g^mI(x, g^{N_2\lvert\mathcal{L}\rvert}x, g^{m}L_0)$.
Then $L_2$ contains $\gamma \cup (g^m\gamma) \cup (g^{2m}\gamma)$.

Inductively, let us construct a geodesic $L_i$ from $L_{i - 1}$ by substituting the segment $g^{(i - 1)m}I(x, g^{N_i\lvert\mathcal{L}\rvert}x, L_0)$ to $g^{(i - 1)m}I(x, g^{N_i\lvert\mathcal{L}\rvert}x, g^{m}L_0)$ for some sufficiently large $N_i < N_{i - 1}$.
Then $L_{\lvert \mathcal{L} \rvert}$ contains $\gamma \cup (g^m\gamma) \cup \dots \cup (g^{\lvert\mathcal{L}\rvert m} \gamma)$.
On the other hand, since $g^{\lvert\mathcal{L}\rvert}L_0 = L_0$, the geodesic $L_{\lvert \mathcal{L} \rvert}$ follows $L_0$ except for $g^m\gamma \cup \dots \cup g^{(\lvert\mathcal{L}\rvert - 1)m}\gamma$.

By the above, $J:=(g^m\gamma) \cup \dots \cup (g^{\lvert\mathcal{L}\rvert m} \gamma)$ is a geodesic segment joining $x$ and $g^{\lvert\mathcal{L}\rvert m} x$.
At last, let us construct the geodesic from $L_{\lvert \mathcal{L} \rvert}$ by substituting $ g^{\ell \lvert \mathcal{L} \rvert m} I(x, g^{\lvert \mathcal{L} \rvert m} x, L_0)$ to $g^{\ell \lvert \mathcal{L} \rvert m} J$ for all $\ell \in \mathbb{Z}$.
Then the result is exactly $L$; therefore, this construction implies $L$ is a geodesic.
\end{proof}

From Lemma \ref{lem:refinement_Bowditch}, we deduce the next theorem.

\begin{thm} \label{thm:rational_length}
Let $G$ be a group acting simplicially on a $\delta$-hyperbolic graph $\mathcal{G}$.
\begin{enumerate}
\item \label{enum1:rational_length} If the action of $G$ has the finiteness property, then $\operatorname{Spec}(G, \mathcal{G})$ consists of rational numbers.
\item \label{enum2:rational_length} If the action of $G$ has the $\kappa$-finiteness property for some positive integer $\kappa$, then $\operatorname{Spec}(G, \mathcal{G})$ consists of fractions of denominator at most $\kappa$.
\end{enumerate}
\end{thm}

\begin{proof}
For every loxodromic $g$, the $m$-th power of $g$ preserves a geodesic for some $0< m \leq \kappa$ by Lemma \ref{lem:refinement_Bowditch}.
Because $g$ acts simplicially, $\tau(g^m)$ is an integer so that $\tau(g) = \tau(g^m) / m$ is a rational number of denominator $m$.
Therefore, the statements \eqref{enum1:rational_length} and \eqref{enum2:rational_length} hold.
\end{proof}

%
%

\section{Rational length spectrum: general case} \label{sec:general} 

In this section, we show that the right-angled Artin group actions on the extension graphs satisfy the finiteness property in Theorem \ref{thm:generalcasefiniteness}. Here we  deal with the general case where the finiteness constant depends on the element, which is the first half of Theorem \ref{thm:main1}.

Let $\Gamma$ denote a \emph{finite connected} simplicial graph, and let $A(\Gamma)$ be the right-angled Artin group of $\Gamma$.
For each vertex $v$ on $\Gamma$, we write $\operatorname{st}_\Gamma(v)$ as the star of $v$ on $\Gamma$, that is, the induced graph of the closed $1$-neighborhood of $v$ on $\Gamma$.
The extension graph of $\Gamma$ is written by $\Gamma^e$ with the edge metric $d_{\Gamma^e}$.
For a vertex $x \in \Gamma^e$, the star of $x$ is written as $\operatorname{st}_{\Gamma^e}(x)$.

A power of a vertex (for instance, $v^n \in A(\Gamma)$) is called a syllable.
For an element $g \in A(\Gamma)$, the syllable length of $g$, denoted by $\| g \|_{\rm syl}$, is the smallest number of syllables, the product of which is $g$.
A syllable decomposition of an element $g \in A(\Gamma)$ is a word decomposition $s_n \dots s_1$ of $g$ with syllables $s_i$ and $n = \| g \|_{\rm syl}$.

We regard $\Gamma$ as a subgraph of its extension graph $\Gamma^e$ by the inclusion $v \mapsto v$ for vertices $v \in \Gamma$.
In this sense, for an element $g \in A(\Gamma)$, the subgraph $\Gamma^g$ is the conjugation of $\Gamma$ by $g$.
For each vertex $x \in \Gamma^e$, we write $\operatorname{st}_{\Gamma^e}(x)$ as the induced graph of the closed $1$-neighborhood of $x$.
For a vertex $v \in \Gamma$, the notation $\operatorname{st}_\Gamma(v)$ denotes the star of $v$ on $\Gamma$, which is equal to $\operatorname{st}_{\Gamma^e}(v) \cap \Gamma$.

We first start with a basic lemma about geodesics between two vertices of the extension graphs.

\begin{lem} \label{lem:suc}
For a connected simplicial graph $\Gamma$, the following holds.
\begin{enumerate}
\item \label{en2:geo} For all $x, y \in \Gamma \subset \Gamma^e$ and $g, h \in A(\Gamma)$, we have $d_{\Gamma}(x, y) \leq d_{\Gamma^e}(x^g, y^h)$.
\item \label{en1:geo} A geodesic lying on $\Gamma$ is a geodesic of $\Gamma^e$.
\end{enumerate}
\end{lem}

\begin{proof}
\noindent \eqref{en2:geo}
Let $\{e_1, \dots, e_n\}$ be the geodesic path of edges from $x$ to $y^{hg^{-1}}$.
For each $i$, there exists $g_i$ such that $e_i^{g_i}$ is contained in $\Gamma$.
And $\{e_1^{g_1}, \dots, e_n^{g_n}\}$ forms a path from $x$ to $y$ on $\Gamma$.
So we have $d_{\Gamma}(x, y) \leq n = d_{\Gamma^e}(x, y^{hg^{-1}}) = d_{\Gamma^e}(x^g, y^h)$.
\vspace{2mm}

\noindent \eqref{en1:geo}
By this way, every path from $x$ to $y$ can be deformed to a path on $\Gamma$ with the same length.
Therefore, there exists a path on $\Gamma$ joining $x$ to $y$, which has length $d_{\Gamma^e}(x, y)$.
\end{proof}

By this lemma, we can construct a geodesic of $\Gamma^e$ which follows a syllable decomposition.

\begin{prop} \label{prop:geod_decom}
For an element $g \in A(\Gamma)$ and vertices $u, v \in \Gamma$, there exists a syllable decomposition $g = s_n \dots s_1$ and geodesic segments $\delta_0, \dots, \delta_n \subset \Gamma$ such that \begin{enumerate} \item $\delta_{i-1}$ and $\delta_{i}^{s_{i}}$ share an endpoint for each $i$ and \item the concatenation of $\delta_0, \delta_1^{s_1}, \dots, \delta_n^{s_n \dots s_1}$ is a geodesic joining $u$ and $v^g$. \end{enumerate}
As a corollary, $\Gamma \cup (\bigcup_{i = 1}^n \Gamma^{s_i \dots s_1})$ contains a geodesic joining $u$ and $v^g$.
\end{prop}

\begin{proof}
Let $\gamma$ be a geodesic path from $u$ to $v^g$.
For each $i$, let $v_i \in \Gamma$ be the vertex supporting $s_i$, and write $z_i$ as the vertex $v_i^{s_i \dots s_1}$.
By Lemma \ref{lem:sylsep}, if $\operatorname{st}_{\Gamma^e}(z_i)$ contains neither $u$ nor $v^g$, then this separates $u$ from $v^g$.
So $\gamma$ passes through $\operatorname{st}_{\Gamma^e}(z_i)$ for each $i = 1, \dots, n$.

For each $i$, let $y_i$ be the vertex which $\gamma$ first intersects in $\operatorname{st}_{\Gamma^e}(z_i)$.
We claim that the inequality $\gamma^{-1}(y_i) \leq \gamma^{-1}(y_j)$ holds if $i < j$ and $[s_i, s_j] \neq 1$.
By Lemma \ref{lem:suc}\eqref{en1:geo}, one has $d_{\Gamma^e}(z_i, z_j) \geq d_{\Gamma}(v_i, v_j) \geq 2$.
Since $z_i$ lies on $\Gamma \cup \bigcup_{l = 1}^j\Gamma^{s_l \dots s_1}$, $\operatorname{st}_{\Gamma^e}(z_i)$ contains $u$ or separates $z_j$ from $u$ by Lemma \ref{lem:sylsep}.
So $\gamma$ cannot pass through $y_j$ before it intersects $\operatorname{st}_{\Gamma^e}(z_i)$.
Therefore, the claim is satisfied.

In the syllable decomposition of $g$, let us transpose syllables $s_i$ and $s_j$ repeatedly whenever a subword $s_j s_i$ of the decomposition satisfies the inequalities $i < j$ and $\gamma^{-1}(y_i) > \gamma^{-1}(y_j)$.
Because of the above claim, such transpositions occur only if syllables commute.
So the composition of these transpositions gives another syllable decomposition $g = s_{\sigma(n)} \dots s_{\sigma(1)}$ satisfying $\gamma^{-1}(y_{\sigma(i)}) \leq \gamma^{-1}(y_{\sigma(j)})$ for all $i < j$.

Passing to the above permutation, we suppose that the syllable decomposition $g = s_n \dots s_1$ has the property that $\gamma^{-1}(y_i) \leq \gamma^{-1}(y_j)$ for all $i < j$.

For each $i \in \{1, \dots, n\}$, let $x_i$ be the \emph{unique} vertex of $\operatorname{st}_\Gamma(v_i)^{s_i \dots s_1} \cap \{ y_i^h \mid h \in A(\Gamma) \}$.
And write $x_0 := u$ and $x_{n + 1} := v^g$.
Then for each $i \in \{0, \dots, n\}$, because both $x_i$ and $x_{i + 1}$ are contained in $\Gamma^{s_i \dots s_1}$, we have $d_{\Gamma^e}(x_i, x_{i + 1}) \leq d_{\Gamma^e}(y_i, y_{i + 1})$ by Lemma \ref{lem:suc}\eqref{en2:geo}.
So we have $d_{\Gamma^e}(u, v^g) = \sum_{i = 0}^n d_{\Gamma^e}(y_i, y_{i + 1}) = \sum_{i = 0}^n d_{\Gamma^e}(x_i, x_{i + 1})$.

By Lemma \ref{lem:suc}\eqref{en1:geo}, for each $i \in \{0, \dots, n\}$, we can take a geodesic $L_i$ joining $x_i$ and $x_{i + 1}$, which lies on $\Gamma^{s_i \dots s_1}$.
Therefore, the concatenation of $L_0, \dots, L_n$ is a geodesic joining $u$ and $v^g$ and is contained in $\Gamma \cup \bigcup_{i = 1}^n \Gamma^{s_i \dots s_1}$.
\end{proof}

\subsection{The weak convexity of $\Lambda_g$}

For an element $g \in A(\Gamma)$, let $\mathcal{S}(g)$ denote the collection of all syllable decompositions of $g$.
We define the subgraph $$\Lambda_g := \Gamma \cup \left( \bigcup_{s_n \dots s_1 \in \mathcal{S}(g)} \bigcup_{i = 1}^n \Gamma^{s_i \dots s_1} \right).$$

Now we show $\Lambda_g$ is weakly convex.
Choose two vertices $x, y \in \Lambda_g$.
If either $x$ or $y$ belongs to $\Gamma$, then there exists a geodesic joining $x$ and $y$ in $\Lambda_g$ by Proposition \ref{prop:geod_decom}.
Suppose neither $x$ nor $y$ is contained in $\Gamma$.
Then there exist two syllable decompositions $g = s_n \dots s_1 = s_{\sigma(n)} \dots s_{\sigma(1)}$ and $i_0, j_0 \in \{1, \dots, n\}$ such that $x \in \Gamma^{s_{i_0} \dots s_1}$ and $y \in \Gamma^{s_{\sigma(j_0)} \dots s_{\sigma(1)}}$.

\begin{prop} \label{prop:thick_convex}
For every $g \in A(\Gamma)$, we have $\Lambda_g$ is weakly convex.
\end{prop}

\begin{proof}
Choose two vertices $x, y \in \Lambda_g$.
To show $\Lambda_g$ is weakly convex, we need only to show that $\Lambda_g$ contains a geodesic joining $x$ and $y$.
If $x$ or $y$ is contained in $\Gamma$, then $\Lambda_g$ contains a geodesic joining $x$ and $y$ due to Proposition \ref{prop:geod_decom}.
Suppose $\{x, y\}$ is disjoint from $\Gamma$.
By definition, there exist two syllable decompositions $g = s_n \dots s_1 = s_{\sigma(n)} \dots s_{\sigma(1)}$ and $i_0, j_0 \in \{1, \dots, n\}$ such that $x \in \Gamma^{s_{i_0} \dots s_1}$ and $y \in \Gamma^{s_{\sigma(j_0)} \dots s_{\sigma(1)}}$.

If $s_{i_0} \dots s_1$ can be reformed to a syllable subword of $s_{\sigma(j_0)} \dots s_{\sigma(1)}$, then $\Lambda_g^{(s_{i_0} \dots s_1)^{-1}}$ contains a geodesic joining $x^{(s_{i_0} \dots s_1)^{-1}}$ and $y^{(s_{i_0} \dots s_1)^{-1}}$ by Proposition \ref{prop:geod_decom}.
So $\Lambda_g$ contains a geodesic joining $x$ and $y$.
Similarly, if $s_{\sigma(j_0)} \dots s_{\sigma(1)}$ can be reformed to a syllable subword of $s_{i_0} \dots s_1$, then $\Lambda_g$ contains a geodesic joining $x$ and $y$.

Assume each of $s_{i_0} \dots s_1$ and $s_{\sigma(j_0)} \dots s_{\sigma(1)}$ cannot be reformed to a syllable subword of the other.
By Lemma \ref{lem:additional_lem}, one can write $s_{i_0} \dots s_1 = w_1 w_0$ and $s_{\sigma(j_0)} \dots s_{\sigma(1)} = w_2 w_0$ for some nontrivial commuting syllable subwords $w_1$ and $w_2$ with $\operatorname{supp}(w_1) \oplus \operatorname{supp}(w_2) \subseteq \Gamma$.
Since $x^{(s_{i_0} \dots s_1)^{-1}} \in \Gamma$ and $y^{(s_{i_0} \dots s_1)^{-1}} \in \Gamma^{w_2 w_1^{-1}}$, we are able to apply Proposition \ref{prop:geod_decom} so that a concatenation of conjugations of $\Gamma$ following rightmost syllable subwords of some syllable decomposition of $w_2 w_1^{-1}$ contains a geodesic $L$ joining $x^{(s_{i_0} \dots s_1)^{-1}}$ and $y^{(s_{i_0} \dots s_1)^{-1}}$.

Because $\operatorname{supp}(w_1)$ and $\operatorname{supp}(w_2)$ form a subjoin of $\Gamma$, if a syllable subword $w$ can be reformed to a rightmost syllable subword of $w_2 w_1^{-1}$, then $w$ is written by $t_2 t_1^{-1}$ for some rightmost syllable subwords $t_1$ and $t_2$ of $w_1$ and $w_2$, respectively.
If $L$ intersects $\Gamma^{t_2t_1^{-1}}$, then $L^{w_1w_0}$ intersects $\Gamma^{t_2 (t_1^{-1}w_1) w_0}$.
A reduced form of $t_2 (t_1^{-1}w_1) w_0$ is a rightmost syllable subword of $g$ since $g = w_3 w_2 w_1 w_0 = w_3 (w_2 t_2^{-1}) t_1 t_2 (t_1^{-1}w_1) w_0$.
So we have $\Gamma^{t_2(t_1^{-1}w_1)w_0} \subseteq \Lambda_g$.
In summary, if $L$ intersects $\Gamma^{t_2 t_1^{-1}}$, then $\Gamma^{t_2(t_1^{-1}w_1)w_0} \subseteq \Lambda_g$.

Note $L^{w_1w_0}$ is a geodesic joining $x$ and $y$.
The union of $\Gamma^{t_2(t_1^{-1}w_1)w_0}$ through all rightmost syllable subwords $t_1$ and $t_2$ of $w_1$ and $w_2$, respectively, contains $L^{w_1w_0}$.
Because this union lies on $\Lambda_g$, the geodesic $L^{w_1w_0}$ belongs to $\Lambda_g$, that is, there exists a geodesic joining $x$ and $y$ in $\Lambda_g$.
Therefore, $\Lambda_g$ is weakly convex.
\end{proof}

\subsection{Axial subgraph}
Suppose $g$ is cyclically syllable-reduced loxodromic.
Using the above fact as a building block, we can construct a weakly convex subgraph invariant from a cyclically syllable-reduced loxodromic.
Let $\mathcal{T}_g$ denote the following subgraph: $$\mathcal{T}_g := \bigcup_{m \geq 1} \Lambda_{g^{2m}}^{g^{-m}}.$$
Our goal is to show $\mathcal{T}_g$ is an axial subgraph of $g$ by proving sequel Lemmas \ref{lem:base_T} -- \ref{lem:T_cocompact}.
The next lemma is followed by the definition of $\Lambda_g$.

\begin{lem} \label{lem:base_T}
For all $m \geq 0$ and $n \geq 1$, the next two inclusions hold: \begin{enumerate} \item $\Lambda_{g^n} \subseteq \Lambda_{g^{m+n}}$; and \item  $\Lambda_{g^n}^{g^{m}} \subseteq \Lambda_{g^{m + n}}$. \end{enumerate}
\end{lem}

Lemma \ref{lem:base_T} induces the sequence $\{ \Lambda_{g^{2m}}^{g^{-m}} \}_{m \geq 1}$ is increasing. That is,

\begin{lem} \label{lem:increasing_lambda}
For all $m \geq 1$, we have $\Lambda_{g^{2m}}^{g^{-m}} \subseteq \Lambda_{g^{2(m+1)}}^{g^{-(m+1)}}$.
\end{lem}

\begin{proof}
Lemma \ref{lem:base_T} implies the inclusion $\Lambda_{g^{2m}}^{g} \subseteq \Lambda_{g^{2m + 1}} \subseteq \Lambda_{g^{2m+2}}$.
After translating these subgraphs by $g^{-(m+1)}$, we obtain the inclusion of the statement.
\end{proof}

Recall a subgraph is said to be weakly convex if the inclusion is an isometric embedding.
For all $m \geq 1$, by Proposition \ref{prop:thick_convex}, $\Lambda_{g^{2m}}^{g^{-m}}$ is weakly convex.
Combining this fact with Lemma \ref{lem:increasing_lambda}, we deduce $\mathcal{T}_g$ is an increasing union of weakly convex subgraphs.
This implies the following lemma.

\begin{lem} \label{lem:T_weak_convex}
$\mathcal{T}_g$ is weakly convex.
\end{lem}

For each $x \in \mathcal{T}_g$, by Lemma \ref{lem:increasing_lambda}, there exists $m$ such that $x \in \Lambda_{g^{2m}}^{g^{-m}}$.
Using Lemma \ref{lem:base_T}, we get $$x^g \in \Lambda_{g^{2m}}^{g^{-m+1}} \subseteq \Lambda_{g^{2m+2}}^{g^{-m-1}} \subseteq \mathcal{T}_g.$$
So the next lemma holds.

\begin{lem}
$\mathcal{T}_g$ is $\langle g \rangle$-invariant.
\end{lem}

Because $g$ is cyclically syllable-reduced loxodromic, only finitely many elliptic words can be realized by subwords of reduced words representing positive powers of $g$ by Proposition \ref{prop:num_elliptic}.
This regulates the number of intersecting edges on $\mathcal{T}_g$ not to be infinite.
So we obtain the following result.

\begin{lem} \label{lem:local_finite_T}
$\mathcal{T}_g$ is locally finite.
\end{lem}

\begin{proof}
Choose a vertex $x \in \mathcal{T}_g$.
We need only to find an upper bound of the number of edges incident to $x$.
Since $x \in \Lambda_{g^{2m}}^{g^{-m}}$ for some $m \geq 1$, there exists a syllable decomposition $g^{2m} = s_k \dots s_1$ such that $x \in \Gamma^{s_{i} \dots s_1 g^{-m}}$ for some $i \in \{1, \dots, k\}$.

Assume  an edge $e$ on $\mathcal{T}_g$ contains $x$ as an endpoint.
Then $e \subset \Lambda_{g^{2m'}}^{g^{-m'}}$ for some $m' \geq 1$.
So we get another syllable decomposition $g^{2m'} = s_{k'}' \dots s_1'$ satisfying $x \in e \subset \Gamma^{s_{i'}' \dots s_1'g^{-m'}}$ for some $i' \in \{0, \dots, k'\}$.

Since $x$ lies on $\left( \Gamma^{s_{i} \dots s_1 g^{-m}} \right) \cap \left( \Gamma^{s_{i'}' \dots s_1'g^{-m'}} \right)$, the translated image $\Gamma \cap \left( \Gamma^{(s_{i'}' \dots s_1'g^{-m'}) (s_i \dots s_1 g^{-m})^{-1}} \right)$ is also nonempty.
So one of the following holds: the words $s_{i'}' \dots s_1'g^{-m'}$ and $s_i \dots s_1 g^{-m}$ represent an identical element, or a reduced form of the product $(s_{i'}' \dots s_1'g^{-m'}) (s_i \dots s_1 g^{-m})^{-1}$ is a star word.
In conclusion, all edges of $\mathcal{T}_g$ intersecting $x$ are contained in $\bigcup\{ \Gamma^{s_{i'}' \dots s_1' g^{-m'}} \mid (s_{i'}' \dots s_1' g^{-m'}) (s_i \dots s_1 g^{-m})^{-1}~\text{is trivial or a star word} \}$.

Because $g$ is loxodromic, the number of star words realized by subwords of powers of $g$ is finite by Proposition \ref{prop:num_elliptic}.
If $N$ is the number of such star words, then the cardinality of vertices of $\operatorname{st}_{\Gamma^e}(x) \cap \mathcal{T}_g$ is bounded above by $2N \lvert V(\Gamma) \rvert$.
\end{proof}

\begin{lem} \label{lem:T_cocompact}
The action of $\langle g \rangle$ on $\mathcal{T}_g$ is cocompact.
\end{lem}

\begin{proof}
For each $i \in \{1, \dots, n\}$, let $v_i$ denote the vertex supporting $s_i$, and write $z_i := v_i^{s_i \dots s_1}$.
Lemma \ref{lem:z_invariant} says $z_i$ is invariant from a syllable decomposition of $g$.
So we deduce that for every $m \geq 1$ and a syllable decomposition $g^{2m} = s_{2mn}' \dots s_1'$ and for every $j$, the graph $\Gamma^{s_j' \dots s_1'}$ contains $z_\ell^{g^k}$ for some $\ell$ and $k$.

Choose $x \in \mathcal{T}_g$.
Then there exists $m \geq 1$ such that $x^{g^m} \in \Lambda_{g^{2m}}$.
So for some syllable decomposition $g^{2m} = s_{2mn}' \dots s_1'$ and $j$, we have $x^{g^m} \in \Gamma^{s_j' \dots s_1'}$.
By the above, there exists $\ell$ and $k$ so that $z_\ell^{g^k} \in \Gamma^{s_j' \dots s_1'}$.
So we have $d^{\Gamma^e}(x, z_\ell^{g^k}) \leq \operatorname{diam}(\Gamma)$.

This implies that if $\mathcal{N}_{\operatorname{diam}(\Gamma)}(z_i)$ is the closed $\operatorname{diam}(\Gamma)$-neighborhood of $z_i$, then $\bigcup_{m \in \mathbb{Z}}\bigcup_{i = 1}^n \mathcal{N}_{\operatorname{diam}(\Gamma)}(z_i^{g^m})$ contains $\mathcal{T}_g$.
By Lemma \ref{lem:local_finite_T}, the intersection $\mathcal{T}_g \cap ( \bigcup_{i = 1}^n \mathcal{N}_{\operatorname{diam}(\Gamma)}(z_i) )$ is locally finite and bounded.
Then this is compact.
Since the $\langle g \rangle$-orbit of $\mathcal{T}_g \cap ( \bigcup_{i = 1}^n \mathcal{N}_{\operatorname{diam}(\Gamma)}(z_i) )$ covers $\mathcal{T}_g$, therefore, the action of $\langle g \rangle$ on $\mathcal{T}_g$ is cocompact.
\end{proof}

\begin{prop} \label{prop:thick_axial}
For a cyclically syllable-reduced loxodromic $g \in A(\Gamma)$ and $m \geq 1$, the subgraph $\mathcal{T}_g$ is an axial subgraph of $g$.
\end{prop}

\begin{proof}
In conclusion, $\mathcal{T}_g$ is weakly convex $\langle g \rangle$-invariant subgraph where $\langle g \rangle$ acts cocompactly.
Hence, $\mathcal{T}_g$ is an axial subgraph of $g$.
\end{proof}

Every loxodromic is conjugate to a cyclically syllable-reduced loxodromic.
So Proposition \ref{prop:thick_axial} implies every loxodromic has an axial subgraph.

\begin{thm} \label{thm:generalcasefiniteness}
For a connected finite simplicial graph $\Gamma$, the action of the right-angled Artin group $A(\Gamma)$ on the extension graph $\Gamma^e$ satisfies the finiteness property.
\end{thm}

The width of an axial subgraph constructed in Proposition \ref{prop:thick_axial} is dependent on a loxodromic, precisely, on the syllable length of a loxodromic.
Therefore, such an axial subgraph does not have uniform width.

\begin{cor}
Every loxodromic of $A(\Gamma)$ has a rational asymptotic translation length.
\end{cor}

\section{Discrete rational length spectrum: large girth}

In this section, we show that the finiteness constant can be made uniform when the girth of the graph is at least $6$ in Theorem \ref{thm:RAAG_finiteness}. The \emph{girth} of a simplicial graph is the minimum positive length of embedded cycles in the graph, and \emph{from now on suppose $\Gamma$ is a finite connected simplicial graph of girth at least $6$ throughout the section.}

Kim--Koberda \cite[Lemma 3.9]{kim2013embedability} showed that the girth of $\Gamma^e$ is equal to the girth of $\Gamma$.
The following lemma presents an interesting property of $\Gamma^e$ when $\Gamma$ has girth at least $6$.
Compared to Proposition \ref{prop:geod_decom}, the large girth of a graph gives a stronger statement.

\begin{prop} \label{prop:shallow_isom_embed}
Let $g$ be an element of $A(\Gamma)$. Then for every syllable decomposition $g = s_n \dots s_1$ and vertices $u, v \in \Gamma$, some geodesic joining $u$ and $v^g$ is contained in $\Gamma \cup (\bigcup_{i = 1}^n \Gamma^{s_i \dots s_1})$.
\end{prop}

\begin{proof}
By Proposition \ref{prop:geod_decom}, for some syllable decomposition $g = s_n \dots s_1$, the union $\Gamma \cup (\bigcup_{i = 1}^n \Gamma^{s_i \dots s_1})$ contains a geodesic $\gamma$ joining $u$ and $v^g$, moreover, $\gamma$ is a concatenation of geodesic segments $L_0, \dots, L_n$ satisfying $L_0 \subset \Gamma$ and $L_i \subset \Gamma^{s_i \dots s_1}$.
Lemma \ref{lem:syl_express} guarantees it is enough to prove the statement: for a commuting pair $s_{j+1} s_{j}$, the union $\Gamma \cup (\bigcup_{i \neq j}\Gamma^{s_i \dots s_1}) \cup \Gamma^{s_{j+1} (s_{j - 1} \dots s_1)}$ contains a geodesic joining $u$ and $v^g$.
As a matter of fact, our goal is to show $\Gamma \cup (\bigcup_{i \neq j} \Gamma^{s_i \dots s_1})$ contains a geodesic joining $u$ and $v^g$.

If $v_j$ and $v_{j+1}$ are the vertices supporting $s_j$ and $s_{j+1}$, respectively, let us write \begin{center} $z := v_j^{s_j \dots s_1}$ \quad and \quad $z' = v_{j+1}^{s_{j+1} \dots s_1}$. \end{center}
Note $z$ and $z'$ are adjacent so that $\operatorname{st}_{\Gamma^e}(z) \cup \operatorname{st}_{\Gamma^e}(z')$, denoted by $T$, is of diameter at most $3$.
Since the girth of $\Gamma^e$ is larger than $4$, any subpath of $T$ does not induce a cycle, that is, $T$ is an induced subtree.
In addition, $z$ and $z'$ are the only internal vertices of $T$, in other words, the leaves of $T$ are all vertices but $z$ and $z'$.

By Lemma \ref{lem:sylsep}, $T$ contains an endpoint of $\gamma$ or separates $u$ and $v^g$.
Let $x$ be the vertex closest to $u$ in $\gamma \cap T$, and let $y$ be the vertex closest to $v^g$ in $\gamma \cap T$.
Because $\gamma \subset \Gamma \cup (\bigcup_{i = 1}^n \Gamma^{s_i \dots s_1})$, we have $x \in \operatorname{st}_{\Gamma}(v_j)^{s_{j-1} \dots s_1} \subset \Gamma^{s_{j - 1} \dots s_1}$ and $y \in \operatorname{st}_{\Gamma}(v_{j+1})^{s_{j+1} \dots s_1} \subset \Gamma^{s_{j+1} \dots s_1}$.

Let $\delta$ denote the segment of $\gamma$ joining $x$ and $y$.
Note $\gamma \cap \Gamma^{s_j \dots s_1}$ is a concatenation of $L_j$ and some subsegments of $L_{j-1}$ and $L_{j+1}$.
Then $\gamma \cap \Gamma^{s_j \dots s_1}$ is a (connected) segment whose endpoints lie on $T$.
Because the length of $\delta$ is maximal among the lengths of segments connecting two vertices of $\gamma \cap T$, we have $\gamma \cap \Gamma^{s_j \dots s_1} \subseteq \delta$.

If $\delta$ lies on $T$, then $\delta$ belongs to $\Gamma^{s_{j - 1} \dots s_1} \cup \Gamma^{s_{j+1} \dots s_1}$ so that $\Gamma \cup (\bigcup_{i \neq j} \Gamma^{s_i \dots s_1})$ contains $\gamma$.
If $\delta$ does not lie on $T$, the girth of $\Gamma^e$ forces the distance of $x$ and $y$ is three and there exists another segment $\delta'$ that joins $x$ and $y$ and lies on $T$.
Let us substitute $\delta$ to $\delta'$ from $\gamma$.
Then we obtain a same result that $\gamma$ lies on $\Gamma \cup (\bigcup_{i \neq j} \Gamma^{s_i \dots s_1})$.
Therefore, the statement holds.
\end{proof}

From the above, we obtain the following lemma.

\begin{lem} \label{lem:finite_convex}
If $g = s_n \dots s_1$ is a syllable decomposition, then $\Omega := \Gamma \cup (\bigcup_{j = 1}^n \Gamma^{s_j \dots s_1})$ is weakly convex.
That is, for all vertices $x, y \in \Omega$, some geodesic joining $x$ and $y$ is contained in $\Omega$.
\end{lem}

Recall the \emph{Euclidean division} of an integer $m$ by $n$ is the unique equation $m = nq + r$ with $q \in \mathbb{Z}$ and $0 \leq r < n$.
For each $m \in \mathbb{Z}$ with the Euclidean division $m = nq + r$, we write $$g(m) := s_r \dots s_1 g^q.$$
Let $\mathcal{A}_{s_n \dots s_1}$ denote the union of all $\Gamma^{g(m)}$ through all integers $m$, that is, $$\mathcal{A}_{s_n \dots s_1} := \bigcup_{m \in \mathbb{Z}} \Gamma^{g(m)}.$$
Note an axial subgraph of $g$, defined in Section \ref{sec:finiteness}, is a $\langle g \rangle$-invariant weakly convex locally finite subgraph such that the induced action of $\langle g \rangle$ is cocompact.

\begin{prop} \label{prop:conv}
Let $g = s_n \dots s_1$ be a syllable decomposition of a cyclically syllable-reduced loxodromic $g \in A(\Gamma)$.
Then $\mathcal{A}_{s_n \dots s_1}$ is an axial subgraph of $g$.
\end{prop}

\begin{proof}
For each $m, \ell \in \mathbb{Z}$ with $m = nq + r$, we have $$g(m)g^{-\ell} = s_r \dots s_1 g^q g^{-\ell} = g(n(q-\ell)+r) = g(m - n\ell).$$
Then $g$ preserves $\mathcal{A}_{s_n \dots s_1}$ because $\mathcal{A}_{s_n \dots s_1}^g = \bigcup_{m \in \mathbb{Z}}\Gamma^{g(m)g} = \bigcup_{m \in \mathbb{Z}}\Gamma^{g(m + n)} = \mathcal{A}_{s_n \dots s_1}$.
So $\mathcal{A}_{s_n \dots s_1}$ is $\langle g \rangle$-invariant.

For a vertex $x \in \mathcal{A}_{s_n \dots s_1}$ with $x = v^w$, if $m_1 < m_2$ and $x \in \Gamma^{g(m_1)} \cap \Gamma^{g(m_2)}$, then $g(m_2) g(m_1)^{-1}$ is a star subword of $g$ that fixes $v$.
Since $g$ is loxodromic, $m_2 - m_1$ is less than $n$.
Then the cardinality of $\{m \mid x \in \Gamma^{g(m)}\}$ is also less than $n$.
So the number of edges in $\mathcal{A}_{s_n \dots s_1}$, which intersect $x$, is less than $n \cdot \lvert \operatorname{st}_\Gamma(v) \rvert$.
So $\mathcal{A}_{s_n \dots s_1}$ is locally finite.

Consider the following equation: 
\begin{align*}
\mathcal{A}_{s_n \dots s_1} 
&= \bigcup_{m \in \mathbb{Z}}\Gamma^{g(m)} \\
&= \bigcup_{\ell \in \mathbb{Z}} (\Gamma^{g(1 + n\ell)} \cup \dots \cup \Gamma^{g(n + n\ell)}) \\
&= \bigcup_{\ell \in \mathbb{Z}} (\Gamma^{g(1)} \cup \dots \cup \Gamma^{g(n)})^{g^\ell}
\end{align*}
This implies the $\langle g \rangle$-orbit of $\Gamma^{g(1)} \cup \dots \Gamma^{g(n)}$ covers $\mathcal{A}_{s_n \dots s_1}$.
So the action of $\langle g \rangle$ on $\mathcal{A}_{s_n \dots s_1}$ is cocompact.

Since $g$ is cyclically syllable-reduced, $g^\ell = (s_n \dots s_1)^\ell$ is a syllable decomposition for every $\ell \in \mathbb{Z}$.
For each $m \geq 0$, let $\Omega_\ell := \bigcup_{m = 0}^{n\ell} \Gamma^{g(m)}$.
Then we have $$\Omega_{2\ell}^{g^{-\ell}} = \bigcup_{m = 0}^{2n\ell} \Gamma^{g(m)g^{-\ell}} = \bigcup_{m = 0}^{2n\ell} \Gamma^{g(m - n\ell)} = \bigcup_{m = -n\ell}^{n\ell}\Gamma^{m}$$ because $g(m)g^{-\ell} = s_r \dots s_1 g^q g^{-\ell} = g(n(q-\ell)+r) = g(m - n\ell)$.
So we obtain the equation $\mathcal{A}_{s_n \dots s_1} = \bigcup_{\ell \geq 1}\Omega_{2\ell}^{g^{-\ell}}$.
 
By Lemma \ref{lem:finite_convex}, each $\Omega_{2\ell}$ is weakly convex, so is $\Omega_{2\ell}^{g^{-\ell}}$.
Note the sequence $\Lambda_{2\ell}^{g^{-\ell}}$ is an ascending chain with respect to the set inclusion.
Therefore, $\mathcal{A}_{s_n \dots s_1}$ is weakly convex.
\end{proof}

For each $v \in \Gamma$, the \emph{link} of $v$, denoted by $\operatorname{lk}_\Gamma(v)$ is the subgraph induced from $\operatorname{st}_\Gamma(v) \setminus \{ v \}$.
The next lemma is the essence of this section.
The width of the axial subgraph $\mathcal{A}_{s_n \dots s_1}$ is bounded by the size of a link.

\begin{lem} \label{lem:link_separating}
If $s_1$ is a power of a vertex $v_1$, then $\operatorname{lk}_\Gamma(v_1)$ separates the ends of $\mathcal{A}_{s_n \dots s_1}$.
\end{lem}

\begin{proof}
We claim $\Gamma^{g(\ell)} \cap \Gamma^{g(t)} \subseteq \operatorname{st}_\Gamma(v_1)$ for all $\ell \leq 0$ and $t > 0$.
Because $g$ is cyclically syllable-reduced loxodromic, if $t - \ell \geq n$, then $\Gamma^{g(\ell)}$ and $\Gamma^{g(t)}$ are disjoint.
Assume $t - \ell < n$.
Because $g(t)g(\ell)^{-1} = (s_t \dots s_1) (s_n \dots s_{n + 1 + \ell})$ is a syllable decomposition and $t > 0$, one has $v_1 \in \operatorname{supp}(g(t)g(\ell)^{-1})$.

In the meanwhile, for any $x \in \Gamma^{g(\ell)} \cap \Gamma^{g(t)}$, if $u$ denotes the vertex $x^{g(\ell)^{-1}} \in \Gamma$, then $u$ commutes with $g(t)g(\ell)^{-1}$.
Then $v_1 \in \operatorname{supp}(g(t)g(\ell)^{-1}) \subseteq \operatorname{st}_\Gamma(u)$, moreover, $u$ commutes with $g(\ell)$ by Centralizer theorem.
So $x = u^{g(\ell)} = u^{g(t)} = u \in \operatorname{st}_\Gamma(v_1)$.
Therefore, we conclude $\Gamma^{g(\ell)} \cap \Gamma^{g(t)} \subseteq \operatorname{st}_\Gamma(v_1)$, so the claim holds.

Since each of $\bigcup_{\ell \leq 0} (\Gamma^{g(\ell)} - \operatorname{st}_\Gamma(v_1))$ and $\bigcup_{t > 0} (\Gamma^{g(t)} - \operatorname{st}_\Gamma(v_1))$ contains an unbounded component, $\operatorname{st}_\Gamma(v_1)$ is an end-separating subgraph of $\mathcal{A}_{s_n \dots s_1}$.
Note $v_1$ is an isolated point of $\mathcal{A}_{s_n \dots s_1} \setminus \operatorname{lk}_\Gamma(v_1)$.
So the unbounded components of $\mathcal{A}_{s_n \dots s_1} \setminus \operatorname{lk}_\Gamma(v_1)$ are identical to the unbounded components of $\mathcal{A}_{s_n \dots s_1} \setminus \operatorname{st}_\Gamma(v_1)$.
Therefore, $\operatorname{lk}_\Gamma(v_1)$ separates the ends of $\mathcal{A}_{s_n \dots s_1}$.
\end{proof}

It is easy to see that the $\kappa$-finiteness property implies the $\kappa'$-finiteness property for all $\kappa \leq \kappa'$ by definition of the finiteness property given in Section \ref{sec:finiteness}.
Now we are ready to prove our main result of the section.

\begin{thm} \label{thm:RAAG_finiteness}
For a finite connected simplicial graph $\Gamma$ of girth at least $6$, the action of $A(\Gamma)$ on $\Gamma^e$ satisfies the $\kappa$-finiteness property for some positive integer $\kappa = \kappa(\Gamma)$.
Furthermore, the effective value for $\kappa$ is bounded above by the maximum degree of $\Gamma$.
\end{thm}

\begin{proof}
For a cyclically syllable-reduced loxodromic, there exists an axial subgraph of width at most $\kappa$ by Lemma \ref{lem:link_separating}.
Because every loxodormic is conjugate to some cyclically syllable-reduced loxodromic, it also has an axial subgraph of width at most $\kappa$.
Therefore, the action of $A(\Gamma)$ on $\Gamma^e$ satisfies the $\kappa$-finiteness property.
\end{proof}

The following corollary is deduced from Theorem \ref{thm:RAAG_finiteness} combined with Lemma \ref{lem:refinement_Bowditch}.

\begin{cor}\label{cor:RAAG_cyclic_permutation}
Suppose the girth of $\Gamma$ is at least $6$.
If $N$ is the maximal degree of $\Gamma$, then every loxodromic of $A(\Gamma)$ permutes cyclically at most $N$ pairwise disjoint geodesics on $\Gamma^e$.
\end{cor}

Then we can obtain Theorem \ref{thm:main1} by Theorem \ref{thm:RAAG_finiteness} and Theorem \ref{thm:rational_length}.

\begin{proof}[Proof of Theorem \ref{thm:main1}]
The action of $A(\Gamma)$ on $\Gamma^e$ satisfies the $\kappa$-finiteness property by Theorem \ref{thm:RAAG_finiteness}.
For every element $g$ of $A(\Gamma)$, the asymptotic translation length of $g$ is a fraction of denominator $\kappa!$ by Theorem \ref{thm:rational_length}.
Therefore, the statement of Theorem \ref{thm:main1} holds.
\end{proof}

\section{Examples} \label{sec:examples}

In this section, we calculate asymptotic translation lengths and their spectra in several cases using Theorem \ref{thm:rational_length}.
For a finite connected simplicial graph $\Gamma$, let $\operatorname{Spec}(A(\Gamma))$ denote the length spectrum of of $A(\Gamma)$ on $\Gamma^e$.
All other symbols and notations we use in this section are adopted from the front of Section \ref{sec:general}.

%
%
\subsection{Trees} \label{subsec:treegraphs}

By \cite[Lemma 3.5(5), Lemma 3.9]{kim2013embedability}, the connectivity and acyclicity of a graph are retained in its extension graph, respectively.
It deduces that an extension graph of a tree is also a tree.
In the case of trees, we do not need to apply the $\kappa$-finiteness property; every loxodromic has a unique geodesic axis with an integer asymptotic translation length by Bass--Serre theory \cite[Proposition 25]{MR607504}.
On the other hand, because a tree is a bi-partite graph, every closed path on a tree has even length.
From the above, we derive the following proposition.

\begin{prop} \label{thm:translation-tree} 
For a finite simplicial tree $\Gamma$, one has $\operatorname{Spec}(A(\Gamma)) \subseteq 2\mathbb{Z}$.
\end{prop} 

\begin{proof} 
We choose a loxodromic $g \in A(\Gamma)$.
Let $\phi: \Gamma^e \to \Gamma$ be the forgetful graph homomorphism defined by $v^g \mapsto v$.
If $x \in \Gamma^e$ is a vertex on the geodesic axis of $g$ and $\gamma$ is the geodesic path joining $x$ to $x^g$, then $\phi(\gamma)$ is a closed path on $\Gamma$.
Because the length of a closed path of $\Gamma$ is even, so is the length of $\gamma$.
Therefore, $\tau(g)$ is an even integer.
\end{proof} 

\subsection{Cycles} \label{subsec:cyclegraphs}
A star of a vertex in a cycle is a path graph of length $2$.
If a cycle is of even length, then it is a bi-partite graph so that every closed path has even length.
These two facts imply the following proposition.

\begin{prop} \label{prop:spectrum_cycle}
Let $C_k$ be a cycle of length $k$ with $k \geq 6$.
\begin{enumerate}
\item \label{enum:cycles0} If $k$ is even, then we have $\operatorname{Spec}(A(C_k)) \subseteq \mathbb{Z}$.
\item \label{enum:cycles1} If $k$ is odd, then we have $\operatorname{Spec}(A(C_k)) \subseteq \{n/2 \mid n \in \mathbb{Z}\}$. And there exists a loxodromic of non-integer asymptotic translation length.
 \end{enumerate}
\end{prop}

Since the maximum degree of $C_k$ is $2$, every loxodromic permutes at most two geodesics by Corollary \ref{cor:RAAG_cyclic_permutation}.
So the asymptotic translation length of a loxodromic can be expressed as a fraction of denominator $2$.
Hence, $\operatorname{Spec}(A(C_k))$ is a subset of $\{ n/2 \mid n \in \mathbb{Z} \}$.

\subsubsection{Proof of Proposition \ref{prop:spectrum_cycle}\eqref{enum:cycles0}}
For a loxodromic $g$, choose a vertex $v^h$ on a geodesic axis of $g^2$.
If $\gamma$ is a geodesic path joining $v^h$ to $v^{hg}$, then its projective image to $C_k$ is a closed path based at $v$.
Because $k$ is even, $C_k$ is a bi-partite graph so that this closed path has even length, which is equal to the length of $\gamma$.
The length of $\gamma$ is equal to $\tau(g^2)$; therefore, $\tau(g)$ is an integer. $\hfill\blacksquare$

\subsubsection{Proof of Proposition \ref{prop:spectrum_cycle}\eqref{enum:cycles1}}
It is enough to give an example of a loxodromic whose asymptotic translation length is not an integer.
Consider $C_k$ as the Cayley graph of $\mathbb{Z} / k \mathbb{Z}$ with respect to the generator $1$.
For each $a \in \mathbb{Z}$, let $v_a$ be the vertex of $C_k$ corresponding to $a$ modulo $k$, that is, the identification $\cdots = v_{a - k} = v_a = v_{a + k} = \cdots$ holds for each $a$.
Write $l := (k + 1) / 2$.
Then our goal is to show the element $g := v_{kl} v_{(k - 1)l} \dots v_{2l} v_l$ of $A(C_k)$ has asymptotic translation length $k(k - 4) / 2$.

\begin{figure}
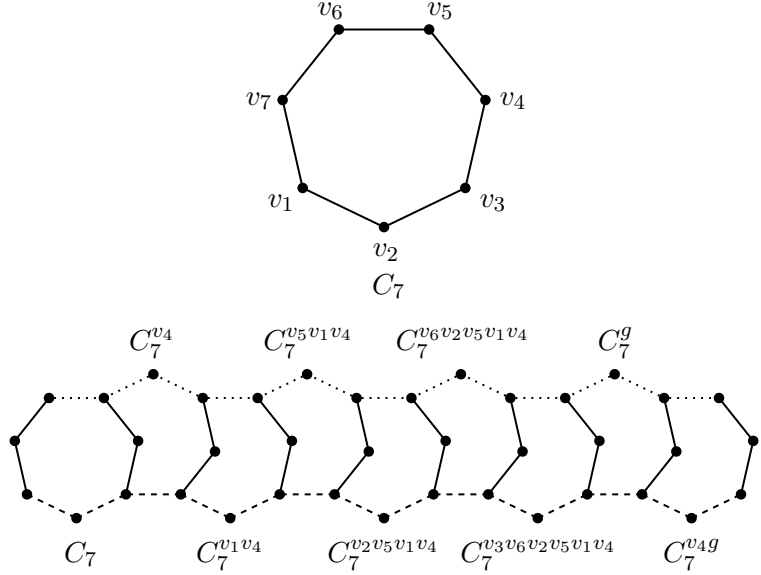
\caption{\label{fig:cycle} The dotted lines on the bottom figure are geodesics preserved by the square of $g := v_7v_3v_6v_2v_5v_1v_4$. So the asymptotic translation length of $g^2$ is $21$, and hence, we have $\tau(g) = 21/2$.}
\end{figure}

By Lemma \ref{cor:RAAG_cyclic_permutation}, there exists a geodesic axis of $g^2$ passing through either $v_{l - 1}$ or $v_{l + 1}$ because $v_{l - 1}$ and $v_{l + 1}$ are all vertices of the link of $v_l$.
Then we have $\tau(g^2) = \min \{ d_{C_k^e}(v_{l - 1}, v_{l - 1}^{g^2}), d_{C_k^e}(v_{l + 1}, v_{l + 1}^{g^2}) \}$.
The remaining part of the proof is the calculation of the distances $d_{C_k^e}(v_{l - 1}, v_{l - 1}^{g^2})$ and $d_{C_k^e}(v_{l + 1}, v_{l + 1}^{g^2})$.

Fix $l_0 \in \{l - 1, l + 1\}$.
Let $\gamma$ is a geodesic joining $v_{l_0}$ and $v_{l_0}^{g^2}$, which lies on the axial subgraph obtained from the syllable decomposition $g = v_{kl} \dots v_l$.
Then $\gamma$ intersects $(\operatorname{lk}_\Gamma(v_{jl}))^{v_{jl} \dots v_{2l} v_l} = \{ v_{jl - 1}^{v_{jl} \dots v_{2l} v_l}, v_{jl + 1}^{v_{jl} \dots v_{2l} v_l} \}$ for each $j \in \{1, \dots, 2k+1\}$ by Lemma \ref{lem:link_separating}.
Then there exists a map $\phi: \{1, \dots, 2k+1\} \to \{-1, 1\}$ such that $\gamma$ contains the vertex $v_{jl + \phi(j)}^{v_{jl} \dots v_{l}}$ for all $j$.

To compute the distance between $v_{l_0}$ and $v_{l_0}^{g^2}$, we need some basic computations of the following lemma.

\begin{lem} \label{lem:cycle_paths}
For each $j \in \mathbb{Z}$, the following equations are satisfied.
\begin{enumerate}
\item \label{enum:cycle_paths1} $d_{C_k}(v_{jl - 1}, v_{(j + 1)l +1}) = k - l - 2$
\item \label{enum:cycle_paths2} $d_{C_k}(v_{jl + 1}, v_{(j + 1)l -1}) = l - 2$
\item \label{enum:cycle_paths3} $d_{C_k}(v_{jl - 1}, v_{(j + 1)l -1}) = k - l = d_{C_k}(v_{jl + 1}, v_{(j + 1)l +1})$
\end{enumerate}
\end{lem}

\begin{proof}
Note that $d_{C_k}(v_{jl - 1}, v_{(j + 1)l +1})$ is the minimum between $(j + 1)l + 1 - (jl + 1) = l$ and $k + jl - 1 - ((j+1)l + 1) = k - l - 2$.
So the equation \eqref{enum:cycle_paths1} holds.
The proofs of \eqref{enum:cycle_paths2} and \eqref{enum:cycle_paths3} are similar to the former.
We leave them as exercises.
\end{proof}

By the equations of the above lemma, we can compute the distance between $v_{l_0}$ and $v_{l_0}^{g^2}$.

\begin{lem}
The inequation holds: $d_{C_k^e}(v_{l_0}, v_{l_0}^{g^2}) \leq k(k - 4)$.
\end{lem}

\begin{proof}
By triangle inequality and Lemma \ref{lem:cycle_paths}\eqref{enum:cycle_paths1}\eqref{enum:cycle_paths2}, we can obtain an upper bound of the distance: If $l_0 = l - 1$, then
\begin{align*}
d_{C_k^e}(v_{l - 1}, v_{l - 1}^{g^2}) &\leq \sum_{j = 1}^{2k} d_{C_k^e}(v_{jl + (-1)^j}^{v_{jl} \dots v_l}, v_{(j + 1)l + (-1)^{j + 1}}^{v_{(j+1)l} \dots v_l}) \\
&= \sum_{j = 1}^{2k} d_{C_k^e}(v_{jl + (-1)^j}, v_{(j + 1)l + (-1)^{j + 1}}) \\ 
&= k(l - 2) + k(k - l - 2) = k(k - 4).
\end{align*}
Similarly, we obtain the inequality $d_{C_k^e}(v_{l + 1}, v_{l - 1}^{g^2}) \leq k(k - 4)$. So the statement holds.
\end{proof}

To know the exact distance of $v_{l_0}$ and $v_{l_0}^{g^2}$, we need the following lemma.

\begin{lem}
We have $d_{C_k^e}(v_{l_0}, v_{l_0}^{g^2}) \geq k (k - 4)$.
\end{lem}

\begin{proof}
Let $b_1, b_2, b_3$ be the numbers defined as follows.
\begin{itemize}
\item $b_1 := \lvert \{ j \in \{1, \dots, 2k\} \mid \phi(j) = -1, \phi(j + 1) = 1 \} \rvert$
\item $b_2 := \lvert \{ j \in \{1, \dots, 2k\} \mid \phi(j) = 1, \phi(j + 1) = -1 \} \rvert$
\item $b_3 := 2k - (b_1 + b_2) = \lvert \{ j \in \{1, \dots, 2k\} \mid \phi(j)\phi(j + 1) = 1 \} \rvert$
\end{itemize}
Then $b_1$ and $b_2$ cannot be larger than $k$ because of the pigeonhole principle.
From Lemma \ref{lem:cycle_paths}, we get the lower bound of the distance:
\begin{align*}
d_{C_k^e}(v_{l_0}, v_{l_0}^{g^2}) &= \sum_{j = 1}^{2k} d_{C_k^e}(v_{jl + \phi(j)}^{v_{jl} \dots v_l}, v_{(j+1)l + \phi(j+1)}^{v_{(j+1)l} \dots v_l}) \\
&= \sum_{j = 1}^{2k} d_{C_k^e}(v_{jl+\phi(j)}, v_{(j + 1)l + \phi(j + 1)}) \\
&= b_1 (k - l - 2) + b_2 (l - 2) + b_3 (k - l) \\
&= k^2 - k - 2b_1 - b_2 \geq k^2 - 4k =  k(k - 4).
\end{align*}
This leads the statement.
\end{proof}

Due to these two inequalities, we have $d_{C_k^e}(v_{l_0}, v_{l_0}^{g^2}) = k(k - 4)$ for all $l_0 \in \{l - 1, l+1\}$.
So the asymptotic translation length of $g^2$ is $k(k - 4)$.
Hence, the asymptotic translation length of $g$ is $k(k - 4) / 2$, which finishes the proof.
$\hfill \blacksquare$

\subsection{Arbitrary denominator} \label{subsec:arbitrarydenominator}
The $\kappa$-finiteness property of each right-angled Artin group does not guarantee the existence of a global denominator for asymptotic translation lengths of all right-angled Artin groups.
In fact, given arbitrary positive integer $k$, we discover a loxodromic of a right-angled Artin group whose asymptotic translation length is expressed as a positive irreducible fraction of denominator $k$.

\begin{prop}[Asymptotic translation length of arbitrary denominator] \label{prop:arbitrary_integer}
For a positive integer $k \geq 2$, there exist a pair of a connected finite simplicial graph $\Gamma_k$ and an element $g \in A(\Gamma_k)$ of syllable length $3$ such that the asymptotic translation length of $g$ on the extension graph $\Gamma_k^e$ is $3 + (1 / k)$.
\end{prop}

\subsubsection{Construction of $\Gamma_k$}

First, we construct a simplicial graph $\Gamma_k$ for $k \geq 2$.
Consider the disjoint union of three star graphs, each of which has $k$ leaves.
Give labels to the $k$-valent vertices of stars as $u, v, t$, respectively.
And let us name the leaves adjacent to $u$ as $u_1, \dots, u_k$, respectively.
Similarly, give names the leaves adjacent to $v$, (\emph{resp.}, the leaves adjacent to $t$) as $v_1, \dots, v_k$, (\emph{resp.}, as $t_1, \dots, t_k$).

Add edges between $u_i$ and $v_i$ for all $i = 1, \dots, k$.
Similarly, join $v_i$ and $t_i$ by an edge for all $i = 1, \dots, k$.
Connect $t_i$ and $u_{i+1}$ by an edge for all $i = 1, \dots, k - 1$, but we join $t_k$ and $u_1$ by a length $2$ path.
$\Gamma_k$ denotes this resulting graph.
As examples, see Figure \ref{subfig:1121} for $k = 2$ and Figure \ref{subfig:de31} for $k = 4$.

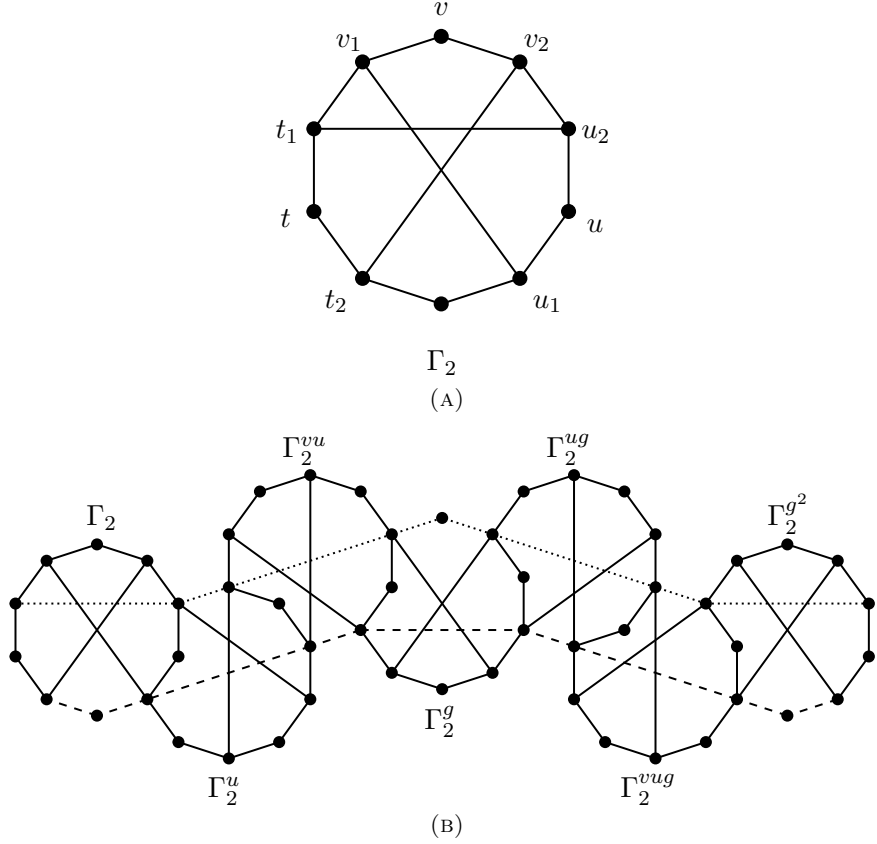
\begin{figure}
\subfloat[][]{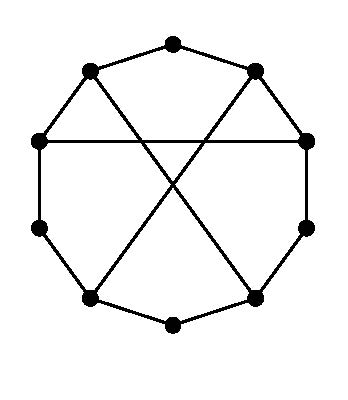 \label{subfig:1121}}\\
\subfloat[][]{
\begingroup%
  \makeatletter%
  \providecommand\color[2][]{%
    \errmessage{(Inkscape) Color is used for the text in Inkscape, but the package 'color.sty' is not loaded}%
    \renewcommand\color[2][]{}%
  }%
  \providecommand\transparent[1]{%
    \errmessage{(Inkscape) Transparency is used (non-zero) for the text in Inkscape, but the package 'transparent.sty' is not loaded}%
    \renewcommand\transparent[1]{}%
  }%
  \providecommand\rotatebox[2]{#2}%
  \newcommand*\fsize{\dimexpr\f@size pt\relax}%
  \newcommand*\lineheight[1]{\fontsize{\fsize}{#1\fsize}\selectfont}%
  \ifx\svgwidth\undefined%
    \setlength{\unitlength}{326.12408976bp}%
    \ifx\svgscale\undefined%
      \relax%
    \else%
      \setlength{\unitlength}{\unitlength * \real{\svgscale}}%
    \fi%
  \else%
    \setlength{\unitlength}{\svgwidth}%
  \fi%
  \global\let\svgwidth\undefined%
  \global\let\svgscale\undefined%
  \makeatother%
  \begin{picture}(1,0.4126285)%
    \lineheight{1}%
    \setlength\tabcolsep{0pt}%
    \put(0,0){\includegraphics[width=\unitlength,page=1]{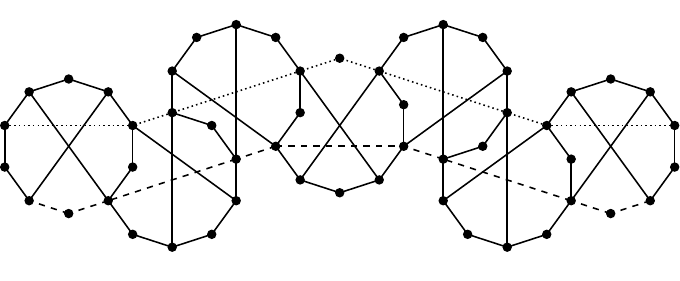}}%
    \put(0.22934422,0.00435035){\makebox(0,0)[lt]{\lineheight{1.25}\smash{\begin{tabular}[t]{l}$\Gamma_2^{u}$\end{tabular}}}}%
    \put(0.08885738,0.31843037){\makebox(0,0)[lt]{\lineheight{1.25}\smash{\begin{tabular}[t]{l}$\Gamma_2$\end{tabular}}}}%
    \put(0.31563928,0.39752687){\makebox(0,0)[lt]{\lineheight{1.25}\smash{\begin{tabular}[t]{l}$\Gamma_2^{vu}$\end{tabular}}}}%
    \put(0.47966765,0.08329141){\makebox(0,0)[lt]{\lineheight{1.25}\smash{\begin{tabular}[t]{l}$\Gamma_2^g$\end{tabular}}}}%
    \put(0.62068585,0.39752687){\makebox(0,0)[lt]{\lineheight{1.25}\smash{\begin{tabular}[t]{l}$\Gamma_2^{ug}$\end{tabular}}}}%
    \put(0.70589793,0.00405329){\makebox(0,0)[lt]{\lineheight{1.25}\smash{\begin{tabular}[t]{l}$\Gamma_2^{vug}$\end{tabular}}}}%
    \put(0.87591796,0.31843033){\makebox(0,0)[lt]{\lineheight{1.25}\smash{\begin{tabular}[t]{l}$\Gamma_2^{g^2}$\end{tabular}}}}%
  \end{picture}%
\endgroup%
 \label{subfig:1122}}
\caption{\label{fig:112} Let $g$ denote the loxodromic $tvu$. The dotted lines in Figure \ref{subfig:1122} are segments of geodesics preserved by $g^2$. So the asymptotic translation length of $g$ is $7/2$.}
\end{figure}

\begin{figure}
\subfloat[][]{
\begingroup%
  \makeatletter%
  \providecommand\color[2][]{%
    \errmessage{(Inkscape) Color is used for the text in Inkscape, but the package 'color.sty' is not loaded}%
    \renewcommand\color[2][]{}%
  }%
  \providecommand\transparent[1]{%
    \errmessage{(Inkscape) Transparency is used (non-zero) for the text in Inkscape, but the package 'transparent.sty' is not loaded}%
    \renewcommand\transparent[1]{}%
  }%
  \providecommand\rotatebox[2]{#2}%
  \newcommand*\fsize{\dimexpr\f@size pt\relax}%
  \newcommand*\lineheight[1]{\fontsize{\fsize}{#1\fsize}\selectfont}%
  \ifx\svgwidth\undefined%
    \setlength{\unitlength}{105.91772966bp}%
    \ifx\svgscale\undefined%
      \relax%
    \else%
      \setlength{\unitlength}{\unitlength * \real{\svgscale}}%
    \fi%
  \else%
    \setlength{\unitlength}{\svgwidth}%
  \fi%
  \global\let\svgwidth\undefined%
  \global\let\svgscale\undefined%
  \makeatother%
  \begin{picture}(1,1.0594252)%
    \lineheight{1}%
    \setlength\tabcolsep{0pt}%
    \put(0.94591499,0.17014471){\makebox(0,0)[lt]{\lineheight{1.25}\smash{\begin{tabular}[t]{l}$u$\end{tabular}}}}%
    \put(0.47668569,0.98677155){\makebox(0,0)[lt]{\lineheight{1.25}\smash{\begin{tabular}[t]{l}$v$\end{tabular}}}}%
    \put(-0.00355075,0.17014463){\makebox(0,0)[lt]{\lineheight{1.25}\smash{\begin{tabular}[t]{l}$t$\end{tabular}}}}%
    \put(0.45614548,0.01664202){\makebox(0,0)[lt]{\lineheight{1.25}\smash{\begin{tabular}[t]{l}$\Gamma_4$\end{tabular}}}}%
    \put(0,0){\includegraphics[width=\unitlength,page=1]{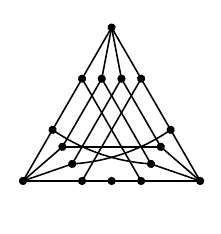}}%
  \end{picture}%
\endgroup%
 \label{subfig:de31}} \\
\subfloat[][This figure is a part of the axial subgraph of $g := tvu$ constructed by Proposition \ref{prop:conv}.]{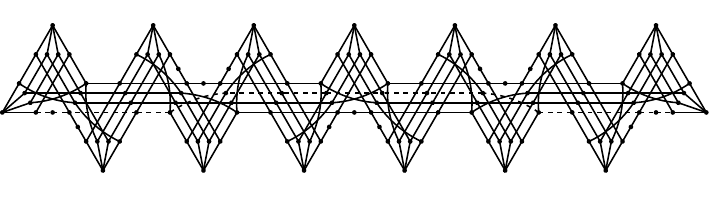 \label{subfig:de32}}
\caption{\label{fig:de3} For the loxodromic $g := tuv$, the dotted line in Figure \ref{subfig:de32} is a segment of a geodesic preserved by $g^4$. So the asymptotic translation length of $g^4$ is $13$. Therefore, we have $\tau(g) = 13 / 4$.}
\end{figure}

We collect the properties of $\Gamma_k$ in the following.

\begin{lem} \label{lem:Graph_k_property}
For each $k \geq 2$, the graph $\Gamma_k$ constructed by the above satisfies the following.
\begin{enumerate}
\item \label{enum:grk_prop1} The links of $u, v, t$ are $\{ u_1, \dots, u_k \}$, $\{ v_1, \dots, v_k \}$, and $\{ t_1, \dots, t_k\}$, respectively.
\item \label{enum:grk_prop2} $d_{\Gamma_k}(u_i, v_j) = \begin{cases} 1 & \text{if}~ i = j, \\ 2 & \text{if}~ i = j + 1, \\ 3, & \text{otherwise}, \end{cases}$ and $d_{\Gamma_k}(v_i, t_j) = \begin{cases} 1 & \text{if}~ i = j, \\ 2 & \text{if}~ i = j + 1, \\ 3, & \text{otherwise}. \end{cases}$
\item \label{enum:grk_prop3} $d_{\Gamma_k}(t_i, u_j) = \begin{cases} 1 & \text{if}~ i + 1 = j, \\ 2 & \text{if}~ i = j ~ \text{or} ~ (i, j) = (k, 1), \\ 3, & \text{otherwise}. \end{cases}$
\end{enumerate}

\end{lem}

\subsubsection{Proof of Proposition \ref{prop:arbitrary_integer}}

Our goal is that the loxodromic $g := tvu$ has asymptotic translation length $3 + (1/k)$.
Because $g$ is cyclically syllable-reduced, the axial subgraph $\mathcal{A}_g$ obtained by the decomposition $g = tvu$ exists by Proposition \ref{prop:conv}.
By Theorem \ref{thm:rational_length}, the loxodromic $g$ permutes cyclically at most $k$ geodesics on $\mathcal{A}_g$.
So one of $g, g^2, \dots, g^k$ preserves a geodesic on $\mathcal{A}_g$.
By Lemma \ref{lem:link_separating}, such a geodesic passes through one of $u_1, \dots, u_k$.
So the asymptotic translation length of $g$ is equal to $\min_{i, n \in \{1, \dots, k\}}d_{\Gamma_k^e}(u_i, u_i^{g^n})/n$. 

Before calculating the distance between $u_i$ and $u_i^{g^n}$, we compute several distances in the following lemma.

\begin{lem} \label{lem:uiuj_distance}
For $i, j \in \{1, \dots, k\}$, the following statements hold.
\begin{enumerate}
\item \label{enum:uiuj1} $d_{\Gamma_k^e}(u_i, u_j^g) = 3$ if and only if $i + 1 = j$.
\item \label{enum:uiuj2} $d_{\Gamma_k^e}(u_i, u_j^g) = 4$ if and only if either $i = j$ or $(i, j) = (k, 1)$.
\item \label{enum:uiuj3} $d_{\Gamma_k^e}(u_i, u_j^g) = 5$, otherwise.
\end{enumerate}
\end{lem}

\begin{proof}
Because $(\operatorname{lk}_{\Gamma_k}(v))^u$ and $(\operatorname{lk}_{\Gamma_k}(t))^{vu}$ separate $u_i$ from $u_j^g$, there exist vertices $v_a \in \operatorname{lk}_{\Gamma_k}(v)$ and $t_b \in \operatorname{lk}_{\Gamma_k}(t)$ such that \begin{align*} d_{\Gamma_k^e}(u_i, u_j^g) &= d_{\Gamma_k^e}(u_i, v_a^u) + d_{\Gamma_k^e}(v_a^u, t_b^{vu}) + d_{\Gamma_k^e}(t_b^{vu}, u_j^g) \\ &= d_{\Gamma_k}(u_i, v_a) + d_{\Gamma_k}(v_a, t_b) + d_{\Gamma_k}(t_b, u_j). \end{align*}
Because $d_{\Gamma_k}(u_i, v_a)$, $d_{\Gamma_k}(v_a, t_b)$, and $d_{\Gamma_k}(t_b, u_j)$ are positive, we obtain the lower bound $d_{\Gamma_k^e}(u_i, u_j^g) \geq 3$.
\vspace{2mm}

\noindent\eqref{enum:uiuj1}
If $i + 1 = j$, then we have $d_{\Gamma_k^e}(u_i, u_{i + 1}^g) \leq d_{\Gamma_k}(u_i, v_i) + d_{\Gamma_k}(v_i, t_i) + d_{\Gamma_k}(t_i, u_{i+1}) = 3$ by the triangle inequality and Lemma \ref{lem:Graph_k_property}\eqref{enum:grk_prop2}.
So we have $d_{\Gamma_k^e}(u_i, u_{i+1}^g) = 3$.
Conversely, if $d_{\Gamma_k^e}(u_i, u_j^g) = 3$, then the equality $d_{\Gamma_k}(u_i, v_a) = d_{\Gamma_k}(v_a, t_b) = d_{\Gamma_k}(t_b, u_j) = 1$ holds.
By Lemma \ref{lem:Graph_k_property}\eqref{enum:grk_prop2}, we have $a = i = b$ and $j = b + 1 = i + 1$.
\vspace{2mm}

\noindent\eqref{enum:uiuj2}
By \eqref{enum:uiuj1}, if $i \geq j$, then we have $d_{\Gamma_k^e}(u_i, u_j^g) \geq 4$.
If $i = j$, then by the triangle inequality, the inequality $d_{\Gamma_k^e}(u_i, u_i^g) \leq d_{\Gamma_k}(u_i, v_i) + d_{\Gamma_k}(v_i, t_i) + d_{\Gamma_k}(t_i, u_i) = 4$ holds.
If $(i, j) = (k, 1)$, we have $d_{\Gamma_k^e}(u_k, u_1^g) \leq d_{\Gamma_k}(u_k, v_k) + d_{\Gamma_k}(v_k, t_k) + d_{\Gamma_k}(t_k, u_1) = 4$.
Conversely, if $d_{\Gamma_k^e}(u_i, u_j^g) = 4$, then one of $d_{\Gamma_k}(u_i, v_a)$, $d_{\Gamma_k}(v_a, t_b)$, and $d_{\Gamma_k}(t_b, u_j)$ in the righthand side of the above equation is $2$, and the others are $1$.
If either $d_{\Gamma_k}(u_i, v_a)$ or $d_{\Gamma_k}(v_a, t_b)$ is two, then by Lemma \ref{lem:Graph_k_property}\eqref{enum:grk_prop2}, we have $b = i - 1$ and $d_{\Gamma_k}(t_b, u_j) = 1$ so that $j = i$.
If $d_{\Gamma_k}(t_b, u_j) = 2$, then $i = a = b$, so by Lemma \ref{lem:Graph_k_property}\eqref{enum:grk_prop3}, we have $i = j$ or $(i, j) = (k, 1)$.
\vspace{2mm}

\noindent\eqref{enum:uiuj3}
If the pair $(i, j)$ does not satisfy \eqref{enum:uiuj1} or \eqref{enum:uiuj2}, then we have $d_{\Gamma_k^e}(u_i, u_j^g) \geq 5$.
And the triangle inequality gives the following: $d_{\Gamma_k^e}(u_i, u_j^g) \leq d_{\Gamma_k}(u_i, u_{j - 1}) + d_{\Gamma_k}(u_{j - 1}, v_{j - 1}) + d_{\Gamma_k}(v_{j - 1}, t_{j - 1}) + d_{\Gamma_k}(t_{j - 1}, u_j) \leq 5$.
Hence the statement follows.
\end{proof}

From the above lemma, we are able to compute the distance between $u_j$ and $u_j^{g^n}$ for each $j$ and $n$.

\begin{lem} \label{lem:medium_distances}
Suppose $k \geq 3$. For each $j \in \{1, \dots, k\}$ and $n \in \{2, \dots k - 1\}$, we have $d_{\Gamma_k^e}(u_j, u_j^{g^n}) \geq 3n + 2$.
\end{lem}

\begin{proof}
Because $(\operatorname{lk}_{\Gamma_k}(u))^{g^l}$ intersects a geodesic joining $u_j$ from $u_j^{g^n}$ for each $l \in \{ 0, \dots, n\}$, the equation $$d_{\Gamma_k^e}(u_j, u_j^{g^n}) = \sum_{l = 0}^{n - 1} d_{\Gamma_k^e}(u_{\phi(l)}^{g^l}, u_{\phi(l + 1)}^{g^{l + 1}}) = \sum_{l = 0}^{n - 1} d_{\Gamma_k^e}(u_{\phi(l)}, u_{\phi(l + 1)}^{g})$$ holds for some map $\phi: \{0, \dots, n \} \to \{ 1, \dots, k \}$ with $\phi(0) = \phi(n) = j$.

Since for each $l$, the distance $d_{\Gamma_k^e}(u_{\phi(l)}, u_{\phi(l +1)}^{g})$ is at least $3$ by Lemma \ref{lem:uiuj_distance}, we have $d_{\Gamma_k^e}(u_j, u_j^{g^n}) \geq 3n$.
If $d_{\Gamma_k^e}(u_j, u_j^{g^n}) = 3n$, then one has $d_{\Gamma_k^e}(u_{\phi(l)}, u_{\phi(l + 1)}^g) = 3$ for all $l$.
Then by Lemma \ref{lem:uiuj_distance}\eqref{enum:uiuj1}, we have $\phi(l) + 1 = \phi(l + 1)$ for all $l$.
This implies $j + n = \phi(0) + n = \phi(n) = j$, which is a contradiction.
So $d_{\Gamma_k^e}(u_j, u_j^{g^n}) > 3n$.

If $d_{\Gamma_k^e}(u_j, u_j^{g^n}) = 3n + 1$, then we have $$d_{\Gamma_k^e}(u_{\phi(l)}, u_{\phi(l+1)}^g) = \begin{cases} 3 & l \neq l_0, \\ 4 & l = l_0 \end{cases}$$ for some $l_0 \in \{0, \dots n-1\}$.
By Lemma \ref{lem:uiuj_distance}\eqref{enum:uiuj2}, we have either $\phi(l_0) = \phi(l_0 + 1)$ or $(\phi(l_0), \phi(l_0 + 1)) = (k, 1)$.
If $\phi(l_0) = \phi(l_0 + 1)$, then we have $j = \phi(n) = \phi(0) + n - 1 = j + n - 1 > j$.
This is a contradiction.
If $\phi(l_0) = k$ and $\phi(l_0 + 1) = 1$, then we have $k = \phi(l_0) = \phi(0) + l_0 = j + l_0$ and $j = \phi(n) = \phi(n - 1) + 1 = \dots = \phi(n - (n - l_0 - 1)) + (n - l_0 - 1) = n - l_0$.
This implies $n = k$, which is also a contradiction.
Therefore, $d_{\Gamma_k^e}(u_j, u_j^{g^n}) \geq 3n + 2$.
\end{proof}

To verify the distance $d_{\Gamma_k^e}(u_j, u_j^{g^k}) / k$ is smaller than the others, we show the following.

\begin{lem} \label{lem:k_upper_bound}
We have $d_{\Gamma_k^e}(u_j, u_j^{g^k}) \leq 3k + 1$ for each $j \in \{1, \dots, k\}$.
\end{lem}

\begin{proof}
By triangle inequality and Lemma \ref{lem:uiuj_distance}\eqref{enum:uiuj1}\eqref{enum:uiuj2}, we have
\begin{align*}
d_{\Gamma_k^e}(u_j, u_j^{g^k}) &\leq \left( \sum_{l = 1}^{k - j} d_{\Gamma_k^e}(u_{j - 1 + l}^{g^{l - 1}}, u_{j + l}^{g^l}) \right) + d_{\Gamma_k^e}(u_k^{g^{k - j}}, u_1^{g^{k - j + 1}}) \\
&\qquad \qquad + \left( \sum_{l = 1}^{j - 1} d_{\Gamma_k^e}(u_{l}^{g^{k - j + l}}, u_{l + 1}^{g^{k - j + l + 1}}) \right) \\
&= \left( \sum_{l = 1}^{k - j} d_{\Gamma_k^e}(u_{j - 1 + l}, u_{j + l}^{g}) \right) + d_{\Gamma_k^e}(u_k, u_1^{g}) + \left( \sum_{l = 1}^{j - 1} d_{\Gamma_k^e}(u_{l}, u_{l + 1}^{g}) \right) \\
&= 3(k - 1) + 4 = 3k + 1.
\end{align*}
So the inequality holds.
\end{proof}

Combining Lemma \ref{lem:uiuj_distance}\eqref{enum:uiuj2}, Lemma \ref{lem:medium_distances}, and Lemma \ref{lem:k_upper_bound}, we obtain the fact that $\tau(g)$ is  equal to $\min_{j \in \{1, \dots,k\}}d_{\Gamma_k^e}(u_j, u_j^{g^k}) / k$.
The following lemma gives the concrete number of $d_{\Gamma_k^e}(u_j, u_j^{g^k}) / k$.

\begin{lem} \label{lem:k_lower_bound}
We have $d_{\Gamma_k^e}(u_j, u_j^{g^k}) \geq 3k + 1$ for each $j \in \{1, \dots, k\}$.
\end{lem}

\begin{proof}
Because $(\operatorname{lk}_{\Gamma_k}(u))^{g^l}$ intersects some geodesic $u_1$ from $u_1^{g^k}$ for each $l \in \{ 0, \dots, k \}$, there exists a map $\phi: \{ 0, \dots, k \} \to \{1, \dots, k\}$ with $\phi(0) = j = \phi(k)$ such that $$d_{\Gamma_k^e}(u_j, u_j^{g^k}) = \sum_{l = 1}^{k} d_{\Gamma_k^e}(u_{\phi(l)}^{g^{l - 1}}, u_{\phi(l + 1)}^{g^l}) = \sum_{l = 1}^{k} d_{\Gamma_k^e}(u_{\phi(l)}, u_{\phi(l + 1)}^{g}).$$
By Lemma \ref{lem:uiuj_distance}, we have $d_{\Gamma_k^e}(u_j, u_j^{g^k}) \geq 3k$.
If $d_{\Gamma_k^e}(u_j, u_j^{g^k}) = 3k$, then for every $l \in \{1, \dots, k\}$, we have $d_{\Gamma_k^e}(u_{\phi(l)}, u_{\phi(l + 1)}^g) = 3$. Then by Lemma \ref{lem:uiuj_distance}, we have $j = \phi(k) = \phi(k - 1) + 1 = \dots = \phi(0) + k = j + k$, which is a contradiction.
Therefore, we have $d_{\Gamma_k^e}(u_j, u_j^{g^k}) \geq 3k + 1$.
\end{proof}

By Lemma \ref{lem:k_upper_bound} and Lemma \ref{lem:k_lower_bound}, we have $d_{\Gamma_k^e}(u_j, u_j^{g^k}) = 3k + 1$ for every $j \in \{1, \dots, k\}$.
By Lemma \ref{lem:uiuj_distance}\eqref{enum:uiuj2}, we obtain that $d_{\Gamma_k^e}(u_j, u_j^g) = 4$ is strictly larger than $d_{\Gamma_k^e}(u_j, u_j^g) / k = (3k + 1) / k$ for all $j \in \{1, \dots k\}$.
We derive from Lemma \ref{lem:medium_distances} the fact that for every $n \in \{ 2, \dots, k - 1 \}$, the distance $d_{\Gamma_k^e}(u_j, u_j^{g^n})$ is larger than $(3k + 1) / k$.
Therefore, we have $\tau(g)  = d_{\Gamma_k^e}(u_j, u_j^{g^k}) = 3 + (1/k)$.
$\hfill\blacksquare$

\subsection{Small syllable length}
Consider a loxodromic of syllable length $2$ for an arbitrary finite simplicial connected graph.
In this case, such a loxodromic is always cyclically syllable-reduced and has a unique syllable decomposition.

\begin{prop} \label{prop:small}
Every loxodromic $g$ of syllable length $2$ has an integer asymptotic translation length.
Precisely, if the support of $g$ is $\{ v_1, v_2 \}$, then we have $\tau(g) = 2d_\Gamma(v_1, v_2) - 4$.
\end{prop}

\begin{proof}
Let $g := s_2s_1$ be a syllable decomposition of $g \in A(\Gamma)$ with $\{ v_i \} = \operatorname{supp}(s_i)$ for each $i$.
If $\gamma$ is a geodesic joining $\operatorname{st}_\Gamma(v_1)$ and $\operatorname{st}_\Gamma(v_2)$ which gives the smallest distance between these stars, then the length of $\gamma$ is equal to $d_\Gamma(v_1, v_2) - 2$.
In addition, because $g$ is loxodromic, the distance of $v_1$ and $v_2$ is more than $2$, so that $\gamma$ has positive length.
Let $u_i$ be the endpoint of $\gamma$ lying on $\operatorname{st}_\Gamma(v_i)$.
Write $L := \bigcup_{q \in \mathbb{Z}}(\gamma \cup \gamma^{s_1})^{g^q}$.

First, we claim that $L$ is a bi-infinite path.
Since $L$ is a quasi-axis of $g$, it is unbounded.
Note that the $\langle g \rangle$-orbit of $\gamma \cup \gamma^{s_1}$ covers $L$.
For every integers $q$ and $r$, if $\lvert q - r \rvert = 1$, then the segments $(\gamma \cup \gamma^{s_1})^{g^q}$ and $(\gamma \cup \gamma^{s_1})^{g^r}$ share only a vertex; therefore, $L$ is connected.
If $\lvert q - r \rvert > 1$, then $(\gamma \cup \gamma^{s_1})^{g^q}$ is disjoint from $(\gamma \cup \gamma^{s_1})^{g^r}$.
This implies all elements of $L$ are bivalent.
In conclusion, $L$ is a connected unbounded $2$-regular graph, that is, $L$ is a bi-infinite path.

Second, we claim that $L$ is a geodesic.
For each $Q > 0$, let $\eta$ denote the segment $\bigcup_{q = 0}^{Q} (\gamma \cup \gamma^{s_1})^{g^q}$, which joins $u_2$ and $u_2^{g^{Q}}$.
The length of $\eta$ is $2Q \cdot (d_\Gamma(v_1, v_2) - 2)$.
By Lemma \ref{lem:finite_convex}, there exists a geodesic path $\delta$ joining $u_2$ to $u_2^{g^Q}$ contained in $\bigcup_{q = 1}^Q (\Gamma \cup \Gamma^{s_1})^{g^q}$.

For each $q$, the length of $\delta \cap \Gamma^{g^q}$ is at least $d_\Gamma(v_1, v_2) - 2$ since the intersection of $\delta$ and $\Gamma^{g^q}$ is a geodesic joining $(\operatorname{st}_\Gamma(v_1))^{g^q}$ and $(\operatorname{st}_\Gamma(v_2))^{g^q}$.
Similarly, the length of $\delta \cap \Gamma^{s_1 g^q}$ is also at least $d_\Gamma(v_1, v_2) - 2$.
So the length of $\delta$ is larger than or equal to $2Q \cdot ( d_\Gamma(v_1, v_2) - 2)$.
This implies that $\eta$ is a geodesic segment.

Because $Q$ is arbitrary, the subray of $L$ starting from $u_2$ is geodesic.
Since $L$ is $\langle g \rangle$-invariant, the whole of $L$ is geodesic.
So the claim holds.
Therefore, the asymptotic translation length of $g$ is equal to $2d_\Gamma(v_1, v_2) - 4$.
\end{proof}

Especially, we obtain the upper bound of the minimum positive asymptotic translation length of $A(\Gamma)$ on $\Gamma^e$.

\begin{cor} \label{cor:length_two}
If the diameter of $\Gamma$ is at least $3$, then the minimum positive asymptotic translation length of $A(\Gamma)$ is at most $2$.
\end{cor}

\begin{proof}
For two vertices $u, v \in \Gamma$ with $d_\Gamma(u, v) = 3$, the loxodromic $vu$ has asymptotic translation length $2$ by Proposition \ref{prop:small}.
\end{proof}

\bibliographystyle{alpha} 
\bibliography{raag}

\end{document}